\documentclass[reqno,12pt]{amsart}
\usepackage{amsmath,amstext, amsthm, amssymb,epsfig,color,tikz,mathrsfs}
\usetikzlibrary{positioning}
\numberwithin{equation}{section}
\newtheorem{thm}{Theorem}[section]
\newtheorem{lem}[thm]{Lemma}
\newtheorem{defn}{Definition}

\newtheorem{cor}[thm]{Corollary}

\newtheorem{conj}[thm]{Conjecture}

\newcommand{\be}{\begin{equation}}
\newcommand{\ee}{\end{equation}}
\newcommand{\ba}{\begin{array}}
\newcommand{\ea}{\end{array}}

\newcommand{\im}{\operatorname{Im}} 
\newcommand{\re}{\operatorname{Re}}
\renewcommand{\em}{\it}

\newtheorem{rem}[thm]{Remark}
\newcommand{\bea}{\begin{eqnarray}}
\newcommand{\eea}{\end{eqnarray}}

\begin{document}
\title[A theory of theta functions to the quintic base]{
A theory of theta functions to the quintic base}
\author{Tim Huber}
\address{Department of Mathematics, University of Texas - Pan American, 1201 West University Avenue,  Edinburg, Texas 78539,  USA}





\begin{abstract}
Properties of four quintic theta functions are developed in parallel with those of the classical Jacobi null theta functions. 
The quintic theta functions are shown to satisfy analogues of Jacobi's quartic theta function identity
and counterparts of Jacobi's Principles of Duplication, Dimidiation and Change of Sign Formulas.
The resulting library of quintic transformation formulas is used to describe series multisections for modular forms in terms of simple matrix operations.
These efforts culminate in a formal technique for deducing congruences modulo powers of five for a variety of combinatorial generating functions, including the partition function. Further analysis of the quintic theta functions is undertaken by exploring their modular properties and their connection to Eisenstein series. The resulting relations lead to a coupled system of differential equations for the quintic theta functions.

\end{abstract}

\maketitle





\vspace{-0.3in}
\section{Introduction} \label{s1}
Let $|q| < 1 $ and define the three null theta functions by 
\begin{align} \label{jac0}
\theta_{3}(q) = \sum_{n= - \infty}^{\infty} q^{n^{2}}, \qquad \theta_{4}(q) = \sum_{n = - \infty}^{\infty} (-1)^{n} q^{n^{2}}, \qquad \theta_{2}(q) = \sum_{n=-\infty}^{\infty} q^{(n + 1/2)^{2}}.
\end{align}
In $1829$, Jacobi proved  
  \begin{align} \label{jac}
    \theta_{3}^{4}(q) = \theta_{4}^{4}(q) + \theta_{2}^{4}(q).
  \end{align}
Jacobi was sufficiently impressed with this identity that he referred to it as an ``\textit{aequatio identica satis abstrusa}'' \cite[p. 147]{funnov}. Identity \eqref{jac} has important combinatorial content 
and interesting interpretations for space-time supersymmetry in the physics literature of string theory \cite[p. 35]{MR2151030}. 
The Jacobi theta functions are building blocks of elliptic functions. Their null values are closely associated with modular forms. A plethora of formulae exist for modular forms in terms of null theta functions. These can be found in the literature beginning with Jacobi's work. S. Ramanujan significantly extended the existing analysis by developing a library of parameterizations for Eisenstein series and allied functions in terms of the null theta functions. Ramanujan may have used these formulae to derive modular equations that appear in his notebooks. Some of Ramanujan's claims are variations on classical results, including Ramanujan's technique for deriving parameterizations in terms of theta functions from existing formulae when $q$ is replaced by $q^{2^{n}}$, $n \in \Bbb Z$. These formulas appear in Jacobi's work as the processes of duplication (for positive $n$) and dimidiation (for negative $n$). Ramanujan also developed a coupled system of differential equations for Eisenstein series equivalent to a coupled system Jacobi derived for the theta functions. A great deal of additional results connect theta functions and their generalizations to important objects and applications spanning many disciplines. Together, the extensive work of Ramanujan, Jacobi, and others constitutes the classical theory of theta functions. 

On pages 257--262 of his Second Notebook, Ramanujan sketched elements of theories of theta functions beyond the classical cases. A number of authors \cite{MR1311903,MR2533672} have proven and extended the claims Ramanujan made about these alternative theories of theta functions. The pioneers of the cubic theory, J.M. and P.B. Borwein, discovered a cubic analogue of \eqref{jac} 
\begin{align} \label{borw}
a^{3}(q) = b^{3}(q) + c^{3}(q),
\end{align}
where the cubic theta functions $a(q)$, $b(q)$, and $c(q)$ are defined by
\begin{align}
a(q) &= \sum_{m,n =-\infty}^{\infty} q^{n^{2} + n m + m^{2}},  \quad
b(q) = \sum_{m,n =-\infty}^{\infty} \omega^{n-m} q^{n^{2} + n m + m^{2}}, \label{abq} \\ 
c(q) &= \sum_{m,n =-\infty}^{\infty}  q^{ \left ( n + \frac{1}{3} \right )^{2} + \left ( n + \frac{1}{3} \right ) \left ( m + \frac{1}{3} \right ) + \left ( m + \frac{1}{3} \right )^{2}}, \quad \omega = e^{2 \pi i/3}. \label{cq}
\end{align}

An extensive theory cubic theta functions has been developed with many features paralleling the classical theory. In addition to \eqref{borw}, core aspects of the theory of cubic theta functions include the fact that $a(q)$ and $c(q)$ parameterize Eisenstein series of level one and three; $a(q), b(q)$, and $c(q)$ satisfy a coupled system of nonlinear differential equations \cite{cubic}; and parameterizations in terms of cubic theta functions satisfy corresponding Principles of Triplication and Trimidiation. The latter principles induce formulas in terms of cubic theta functions when $q$ is replaced by $q^{3^{n}}$, $n \in \Bbb Z$. 

Although not explicitly mentioned by Ramanujan, a corresponding theory of theta functions associated with modular forms of level five may be formulated from identities appearing in Ramanujan's Lost Notebook. The goal of the current work is to derive elements of a quintic theory of theta functions with attributes of the cubic and classical theories. We study four quintic theta functions, $A(q), B(q), C(q)$, and $D(q)$ defined by
 \begin{align*}
    A(q) &= q^{1/5} (q;q)_{\infty}^{-3/5} \sum_{n=-\infty}^{\infty} (-1)^{n} q^{(5n^{2} - 3n)/2}, \quad  B(q) = (q;q)_{\infty}^{-3/5} \sum_{n=-\infty}^{\infty} (-1)^{n} q^{(5n^{2} - n)/2}, \\ C(q) &= \frac{(q^{5};q^{5})_{\infty}^{-3/5}}{1 + e^{3 \pi i/5}} \sum_{n = - \infty}^{\infty} e^{3 \pi i n/5} q^{(n^{2} - n)/2}, \quad D(q) =  \frac{(q^{5};q^{5})_{\infty}^{-3/5}}{1 + e^{\pi i/5}}\sum_{n = - \infty}^{\infty} e^{\pi i n/5} q^{(n^{2} - n)/2}.
  \end{align*}
We employ the standard product notation $(a;q)_{n} = \prod_{k=0}^{n-1} (1 - aq^{k})$, and, provided the limit exists, denote $(a;q)_{\infty} = \lim_{n \to \infty} (a;q)_{n}$. The quintic theta functions $A(q)$ and $B(q)$ are closely related to the Rogers-Ramanujan functions $G(q)$ and $H(q)$, where 
\begin{align} \label{rog}
  G(q) := \sum_{n=0}^{\infty} \frac{q^{n^{2}}}{(q;q)_{n}} = \frac{B(q)}{(q;q)_{\infty}^{2/5}}, \qquad H(q) = \sum_{n=0}^{\infty} \frac{q^{n^{2}+n}}{(q;q)_{n}} =  \frac{q^{-1/5}A(q)}{(q;q)_{\infty}^{2/5}}.
\end{align}

The following four quintic theta function identities play an important role in our subsequent analysis. 
The first two relations are quintic analogues of \eqref{jac} and \eqref{borw}.
\begin{thm} \label{mainthm} If $\alpha = \frac{1 + \sqrt{5}}{2}$, $\beta = \frac{1 - \sqrt{5}}{2}$, then 
\begin{align} \label{d00}
   C^{5}(q) &=  B^{5}(q) - \alpha^{5}A^{5}(q), \qquad    D^{5}(q) =   B^{5}(q) - \beta^{5}A^{5}(q), \\
 \label{po}  C(q) &= B(q^{5}) - \alpha A(q^{5}), \qquad \ \   D(q) = B(q^{5}) - \beta A(q^{5}).
\end{align}    
  \end{thm}
Equations \eqref{po} follow from a general decomposition theorem of Ramanujan for theta functions. Equations \eqref{d00} will be derived from product representations for the quintic theta functions and their relationship to Eisenstein series of level five
\begin{align} \label{eisdef}
  E_{k,\chi}(q) = 1 +  \frac{2}{L(1 - k, \chi)} \sum_{n=1}^{\infty} \chi(n) \frac{n^{k-1} q^{n}}{1 - q^{n}},
\end{align}
where $L(1 - k, \chi)$ is the analytic continuation of the associated Dirichlet $L$-series and $\chi(-1) = (-1)^{k}$. Theorem \ref{mainthm} was anticipated by Ramanujan in his Lost Notebook \cite[Entry 1.4.1, pp. 21-22]{andrews05}. He let $t$ denote the Rogers-Ramanujan continued fraction \cite{rogers}
\begin{align} \label{d1}
t := \frac{A(q)}{B(q)} =  q^{1/5}\frac{(q;q^{5})_{\infty} (q^{4}; q^{5})_{\infty}}{(q^{2};q^{5})_{\infty} (q^{3}; q^{5})_{\infty}} =  \frac{q^{1/5}}{1 + \frac{q}{1 + \frac{q^{2}}{1 + \frac{q^{3}}{1+\cdots}}}},
\end{align}
and gave four identities that we will show in Section \ref{s2} are equivalent to Theorem \ref{mainthm}
\begin{align} \label{gj6}
  \frac{1}{\sqrt{t}} + \alpha \sqrt{t} &= q^{- 1/10} \sqrt{\frac{(q;q)_{\infty}}{(q^{5}; q^{5})_{\infty}}} \prod_{n=1}^{\infty} \frac{1}{1 + \alpha q^{n/5} + q^{2 n/5}}, \\ \frac{1}{\sqrt{t}} + \beta \sqrt{t} &= q^{- 1/10} \sqrt{\frac{(q;q)_{\infty}}{(q^{5}; q^{5})_{\infty}}} \prod_{n=1}^{\infty} \frac{1}{1 + \beta q^{n/5} + q^{2 n/5}}, \label{gj7} \\  \label{po1}
  \left ( \frac{1}{\sqrt{t}} \right )^{5} - &\left ( \alpha \sqrt{t} \right )^{5} = \frac{1}{q^{1/2}} \sqrt{\frac{(q;q)_{\infty}}{(q^{5}; q^{5})_{\infty}}} \prod_{n=1}^{\infty} \frac{1}{(1 + \alpha q^{n} + q^{2n})^{5}}, \\ 
  \left ( \frac{1}{\sqrt{t}} \right )^{5} - &\left ( \beta \sqrt{t} \right )^{5} = \frac{1}{q^{1/2}} \sqrt{\frac{(q;q)_{\infty}}{(q^{5}; q^{5})_{\infty}}} \prod_{n=1}^{\infty} \frac{1}{(1 + \beta q^{n} + q^{2n})^{5}}. \label{po2}
\end{align}

In Section \ref{s3}, we show that \eqref{d00}--\eqref{po} induce formulas for series multisections in terms of quintic theta functions. 
We speculate that Ramanujan used such relations in \cite{MR1701582} to decompose the generating function for integer partitions $\sum_{n=0}^{\infty} p(n) q^{n} = 1/(q;q)_{\infty}$ into generating functions corresponding to certain congruence classes modulo $5^{\lambda}$, $\lambda \in \Bbb N$. These calculations motivated his famous congruences modulo five
\begin{align} \label{par}
  p(5^{\lambda}n + \delta_{\lambda}) \equiv 0 \pmod{5^{\lambda}}, \qquad 24 \delta_{\lambda} \equiv 1 \pmod{5^{\lambda}}.
\end{align}
Ramanujan sketched an inductive argument for \eqref{par}, beginning with the identities
\begin{align}
\begin{split}
\sum_{n=1}^{\infty} \left ( \frac{5}{n} \right ) \frac{q^{n}}{(1 - q^{n})^{2}} = q\frac{(q^{5}; q^{5})_{\infty}^{5}}{(q;q)_{\infty}}, 
\end{split}
\begin{split}
\sum_{n=1}^{\infty} p(5n -1) q^{n} = 
5\frac{q(q^{5}; q^{5})_{\infty}^{5}}{(q;q)_{\infty}^{6}}.
\end{split} \tag{1.13a,b}
\end{align} 
Ramanujan derived (1.13b) by replacing $q$ by $q^{1/5}$ in (1.13a) and equating coefficients of $q^{m/5}$, for $m$ in appropriate residue classes $m$ modulo five. He next asserted that
\begin{align}
\sum_{n=0}^{\infty} p(25n +24) q^{n} &= 
 5^{2}  \cdot 63 \frac{(q^{5};q^{5})_{\infty}^{6}}{(q; q)_{\infty}^{7}} + 5^{5} q\cdot 52\frac{(q^{5};q^{5})_{\infty}^{12}}{(q; q)_{\infty}^{13}} + 5^{7} q^{2}\cdot 63 \frac{(q^{5};q^{5})_{\infty}^{18}}{(q; q)_{\infty}^{19}} \\ & + 5^{10} q^{3}\cdot 6 \frac{(q^{5};q^{5})_{\infty}^{24}}{(q; q)_{\infty}^{25}} + 5^{12}q^{4}\frac{(q^{5};q^{5})_{\infty}^{30}}{(q; q)_{\infty}^{31}}. 
\label{bvf1}
\end{align}
Ramanujan indicated that higher-order expansions may be obtained by proceeding similarly, but he did not elaborate on his method for obtaining each successive decomposition.
In Section \ref{partition}, we demonstrate that Ramanujan's formulas may be derived from Theorem \ref{mainthm}. These results are subsumed in the more general analysis of Section \ref{s3}, wherein Theorem \ref{mainthm} is shown to induce quintic analogues of Jacobi's Processes of Duplication and Dimidiation, or, equivalently, systematic procedures 
for deriving representations for series of argument $q^{5^{n}}$, $n \in \Bbb Z$ from initial parameterizations in terms of quintic theta functions. 
These processes lead to matrix characterizations for two quintic operators on the vector space of homogeneous polynomials in $A^{5}(q)$ and $B^{5}(q)$:
\begin{align} \label{vfn}
\pi  \left (  \sum_{n=0}^{\infty} a_{n} q^{n}, \right ) = \sum_{n=0}^{\infty} a_{n} q^{n/5}, \qquad
  \Omega_{5,m} \left (  \sum_{n=0}^{\infty} a_{n} q^{n}, \right ) = \sum_{n=0}^{\infty} a_{5n+m} q^{n}, \quad m = 0, 1, 2, 3, 4.
\end{align}
\begin{thm} \label{gt}
For each homogeneous polynomial $\sum_{k=0}^{d} a_{k}A^{5k}(q)B^{5(d - k)}(q)$ of degree $d$ in $\Bbb C( A^{5}(q), B^{5}(q))$, there exists a $(5d +1) \times (d+1)$ matrix $\mathcal{B}_{d}$ over $\Bbb Z$ such that if $$\mathcal{B}_{d}
  \begin{pmatrix}
    a_{0} & a_{1} & \cdots & a_{d}
  \end{pmatrix}^{T} =
\begin{pmatrix}
    b_{0} & b_{1} & \cdots & b_{5d}
  \end{pmatrix}^{T};$$ $\chi$ denotes the principal character modulo five; and $0 \le m \le 4$, then
\begin{align}
\label{cdq}  \pi \left ( \sum_{k=0}^{d} a_{k}A^{5k}(q)B^{5(d - k)}(q)  \right ) &= \sum_{k=0}^{5d} b_{k}A^{k}(q)B^{(5d- k)}(q),\\
\label{cdq1} \Omega_{5,m} \left ( \sum_{k=0}^{d} a_{k}A^{5k}(q)B^{5(d - k)}(q)  \right ) &= q^{-m/5}\sum_{k=0}^{d-\chi(m)} b_{5k+m}A^{5k+m}(q)B^{(5d- 5k-m)}(q).
\end{align}
\end{thm}

Equation \eqref{cdq1} induces an explicit formulation for the action of the quintic Hecke operator $T_{5} = \Omega_{5,0}$ (see \cite{dish}) on the space of homogeneous polynomials in $A^{5}(q)$ and $B^{5}(q)$ of degree $d$. This space includes the Eisenstein series of level one and those for the Hecke subgroup of level five.  Since Eisenstein series are building blocks of modular forms, their quintic decompositions provide insight into corresponding multisections for
 a much larger class of combinatorial generating functions. In Section \ref{s4}, we focus on representations for multisections of Eisenstein series on the full modular group and those associated with Dirichlet characters modulo five, where we denote
\begin{align} \label{c1}
   \langle \chi_{1,5}(n) \rangle_{n=0}^{4} &= \langle 0, 1, 1, 1, 1 \rangle, \quad  \langle \chi_{2,5}(n) \rangle_{n=0}^{4} = \langle 0, 1, i, -i, -1 \rangle, \\  \langle \chi_{3,5}(n) \rangle_{n=0}^{4} &= \langle 0, 1, -1, -1, 1 \rangle, \quad  \langle \chi_{4,5}(n) \rangle_{n=0}^{4} = \langle 0, 1, -i, i, -1 \rangle. \label{c2}
\end{align} 

Among the most intriguing  of the new formulas developed in Section \ref{s4} are product representations for weight-one Eisenstein series associated with the character $\chi_{4, 5}$.
\begin{thm} \label{fjl} 
\begin{align*}
  \sum_{n=0}^{\infty} \Bigl ( \sum_{d \mid 5 n +4} \chi_{4, 5}(d) \Bigr ) q^{n} &=  \frac{i(q;q)_{\infty} (q^{5}; q^{5})_{\infty}}{(q^{2}; q^{5})_{\infty}^{3}(q^{3}; q^{5})_{\infty}^{3}}, \quad
  \sum_{n=0}^{\infty} \Bigl ( \sum_{d \mid 5 n +3} \chi_{4, 5}(d) \Bigr ) q^{n} =  \frac{(1 + i) (q^{5}; q^{5})_{\infty}^{2}}{(q^{2}; q^{5})_{\infty}(q^{3}; q^{5})_{\infty}}, \\ 
  \sum_{n=0}^{\infty} \Bigl ( \sum_{d \mid 5 n +2} \chi_{4, 5}(d) \Bigr ) q^{n} &= \frac{(1 - i)(q^{5}; q^{5})_{\infty}^{2}}{(q; q^{5})_{\infty}(q^{4}; q^{5})_{\infty}}, \quad 
  \sum_{n=0}^{\infty} \Bigl ( \sum_{d \mid 5 n +1} \chi_{4, 5}(d) \Bigr ) q^{n} = \frac{(q;q)_{\infty} (q^{5}; q^{5})_{\infty}}{(q; q^{5})_{\infty}^{3}(q^{4}; q^{5})_{\infty}^{3}}.
\end{align*}
\end{thm}

Our interpretations for the quintic operators $\pi$ and $\Omega_{5,m}$ will also address conjectures from \cite{qeis} regarding parameterizations for quintic multisections of Eisenstein series on the full modular group. These formulas are representative of quintic multisections for more general classes of Eisenstein series and related combinatorial generating functions.

\begin{cor} \label{hut} Let $A = A(q)$ and $B = B(q)$. For each integer $n$, define
\begin{align} \label{bn}
F_{n}(q) = a_{1,n}& B^{20} +  a_{2,n} B^{15} A^{5} + a_{3,n} B^{10} A^{10} + a_{4,n} B^{5} A^{15} + a_{5,n} A^{20}, \\
  \begin{pmatrix}
    a_{1,n} \\ a_{2,n} \\ a_{3,n} \\ a_{4,n} \\ a_{5,n}
  \end{pmatrix}
 &=
 \begin{pmatrix}
1 & 0 & 0 & 0 & 0 \\ 
1356 & 115 & 10 & 10 & -8 \\ 
1462 & 110 & 15 & -110 & 1462 \\
8 & 10 & -10 & 115 & -1356 \\ 
0 & 0 & 0 & 0 & 1 
 \end{pmatrix}^{n}
 \begin{pmatrix}
   1 \\ 228 \\ 494 \\ -228 \\ 1
 \end{pmatrix}.
\end{align}
Then, for each nonnegative integer $n$, 
\begin{align*}
1 + 240 \sum_{k=1}^{\infty} \Bigl ( \sum_{d \mid 5^{n} k} d^{3} \Bigr ) q^{k} =
 F_{n}(q), \qquad 
 1 + 240 \sum_{k=1}^{\infty} \Bigl ( \sum_{d \mid k} d^{3} \Bigr ) q^{k} &= 
\Omega_{5,0}^{(n)} \circ F_{-n}(q).
\end{align*}

\end{cor}


The utility of Theorem \ref{gt} extends beyond series representable as homogeneous polynomials in the quintic theta functions. 
In Section \ref{gen_part}, we formulate quintic multisections for a general class of series appearing in Ramanujan's Lost Notebook \cite{ramlost} that includes the generating function for ordinary partitions \begin{align} \label{grt}
  \frac{1}{(q;q)_{\infty}^{k}} = \sum_{n=0}^{\infty} p_{k}(n) q^{n}, \qquad k \in \Bbb Z.
\end{align}
For positive $k \in \Bbb N$, these functions generate the number of $k$-component integer multipartitions \cite{MR2462943}, defined as the number of $k$-tuples $(\lambda_{1}, \cdots, \lambda_{k})$ of partitions with
\begin{align}
  |\lambda_{1}| + \cdots + |\lambda_{k}| = n,
\end{align}
Our inductive characterization of multisections for the series in \eqref{grt} is predicated on a pair of parallel quintic parameterizations for Eisenstein series of weight two
\begin{align} \label{ghee}
  \sum_{n=1}^{\infty} \left ( \frac{n}{5} \right ) \frac{q^{n}}{(1 - q^{n})^{2}} = A^{5}(q)B^{5}(q), \qquad 1 - 5 \sum_{n=1}^{\infty} \left ( \frac{n}{5} \right ) \frac{n q^{n}}{1 - q^{n}} = C^{5}(q) D^{5}(q).
\end{align}

Ramanujan derived an equivalent formulation of these identities in \cite[p. 139]{ramlost}, where the right sides are expressed as infinite products. 
In \cite{qeis}, the authors show more generally that the Eisenstein series of level one and five are parameterizable as homogeneous polynomials symmetric in absolute value about the middle terms. From the perspective of modular forms, the symmetry of parameterizations for Eisenstein series in terms of theta functions is surprising. However, the existence of relations between the series is expected, since $A(q)$ and $B(q)$ are known to be weight-$1/5$ modular forms on $\Gamma(5)$ with suitable factors of automorphy \cite{MR1893493}. In Section \ref{modular}, we provide an additional analytic connection between the pairs $A(q), B(q)$ and $C(q), D(q)$ by showing that the Fricke involution, $W_{5}:\tau \mapsto - 1/5 \tau$, induces a natural correspondence.

\begin{thm} \label{fricke}
$W_{5}$ maps $A(q)$ to $\gamma_{1} \sqrt[5]{\tau} C(q)$ and $B(q)$ to $\gamma_{2} \sqrt[5]{\tau}D(q)$, where $\gamma_{1}, \gamma_{2} \in \Bbb C$. 
\end{thm}

Our final contribution to the theory of quintic theta functions appears in Section \ref{s6}, where 
a nonlinear coupled differential system for the quintic theta functions is derived.
\begin{thm} \label{d_quint} Let $\mathscr{P}(q) = E_{2}(q^{5})$, where $ E_{2}(q) = 1 - 24 \sum_{n=1}^{\infty} \bigl ( \sum_{d \mid n} d \bigr ) q^{n}$. Then 
  \begin{align}
    q \frac{d}{dq} A &= \frac{1}{60} A \Bigl (-5 A^{10}-66 A^5 B^5+7 B^{10}+5 \mathscr{P} \Bigr ), \\ 
q \frac{d}{dq} B &= \frac{1}{60} B \Bigl (7 A^{10}+66 A^5 B^5-5 B^{10}+5 \mathscr{P} \Bigr ), \\ 
q \frac{d}{dq} \mathscr{P} &= \frac{5}{12}\left( \mathscr{P}^{2}-B^{20}+ 12 B^{15}A^5 - 14 B^{10}A^{10}- 12 B^5 A^{15}- A^{20}\right). \label{dpp}
  \end{align}
\end{thm}
These equations are quintic counterparts of Ramanujan's differential system for Eisenstein series on the full modular group
{\allowdisplaybreaks \begin{equation}
\label{rdiff1}  q \frac{dE_{2}}{dq} = \frac{E_{2}^{2} - E_{4}}{12}, \qquad  q \frac{dE_{4}}{dq} = \frac{E_{2}E_{4} - E_{6}}{3}, \qquad q \frac{dE_{6}}{dq} = \frac{E_{2}E_{6} - E_{4}^{2}}{2},
\end{equation}}
where $E_{k} = E_{k}(q)$ denote the Eisenstein series on the full modular group, defined by
\begin{align} \label{eis_full}
  E_{2k}(q) = 1 + \frac{2}{\zeta(1 - 2 k)} \sum_{n=1}^{\infty} \frac{n^{2k-1} q^{n}}{1 - q^{n}},
\end{align}
and where $\zeta$ is the analytic continuation of the Riemann $\zeta$-function.
Theorem \ref{d_quint} is induced from a more general differential system \cite{huber_proc} involving the elliptic parameters
  \begin{align} \nonumber
    e_{\alpha}(q) = 1 + & 4\tan( \pi \alpha) \sum_{n=1}^{\infty} \frac{\sin( 2 n  \pi \alpha) q^{n}}{1 - q^{n}}, \quad P_{\alpha}(q) = 1 - 8 \sin^{2}( \pi \alpha) \sum_{n=1}^{\infty} \frac{\cos (2 n \pi \alpha) n q^{n}}{1 - q^{n}},  \\ & Q_{\alpha}(q) = 1- 8\tan( \pi \alpha) \sin^{2}( \pi \alpha) \sum_{n=1}^{\infty} \frac{\sin( 2 n  \pi \alpha) n^{2} q^{n}}{1 - q^{n}}. \label{ser}
  \end{align}
\begin{thm}  \label{coupl} Let $e_{\alpha}(q), P_{\alpha}(q)$, and $Q_{\alpha}(q)$ be defined as in \eqref{ser}. Then for $\alpha \not \equiv 1/2$,
\allowdisplaybreaks{\begin{align}
\label{deqe}  q \frac{d}{dq}e_{\alpha} &= \frac{\csc^{2}(\pi \alpha )}{4} \left (  e_{\alpha} P_{\alpha} - Q_{\alpha} \right ), \\ 
\label{deqp} q \frac{d}{dq}P_{\alpha} &= \frac{\csc^{2}( \pi \alpha)}{4} P_{\alpha}^{2} - \frac{1}{2} \cot^{2}( \pi \alpha) e_{\alpha} Q_{\alpha} + \frac{1}{2}\cot( \pi \alpha) \cot(2\pi \alpha)e_{1 - 2 \alpha} Q_{\alpha}, \\
\nonumber  q \frac{d}{dq}Q_{\alpha} &= \frac{1}{4} Q_{\alpha} P_{\alpha} \csc^{2}( \pi \alpha) + \frac{1}{2} P_{1 - 2 \alpha} Q_{\alpha} \csc^{2}(2\pi  \alpha) - \frac{1}{2} e_{1 - 2 \alpha}^{2} Q_{\alpha} \cot^{2}(2 \pi \alpha) \\ & \qquad + \frac{3}{2} e_{\alpha} e_{1 - 2 \alpha} Q_{\alpha} \cot( \pi \alpha) \cot(2 \pi \alpha) - e_{\alpha}^{2}Q_{\alpha} \cot^{2}( \pi \alpha). \label{deqq} 
\end{align}}
\end{thm}
The series $e_{\alpha}(q)$, $P_{\alpha}(q)$, and $Q_{\alpha}(q)$ are, respectively, normalized versions of the Weierstrass zeta function $\zeta(\theta)$ and its derivatives $\zeta'(\theta) := -\wp(\theta)$ and $\zeta''(\theta)$, where 
\begin{align} \label{wzdef}
\zeta(\theta) &= \frac{1}{2} \cot \frac{\theta}{2} + \frac{\theta}{12} - 2 \theta \sum_{n=1}^{\infty} \frac{n q^{n}}{1 - q^{n}} + 2 \sum_{n=1}^{\infty} \frac{q^{n} \sin n \theta}{1 - q^{n}}.
\end{align}
\section{A catalogue of fundamental quintic relations} \label{s2}
In this section, we establish a number of fundamental identities that constitute the basis for subsequent results of the paper. These relations will prove Theorem \ref{mainthm} and exhibit its equivalence with Ramanujan's identities \eqref{gj6}--\eqref{po2}. The ensuing discussion will make use of Ramanujan's general theta function notation \cite[p. 34]{berndt91} 
 \begin{align} \label{ramt}
  f(a, b) = \sum_{n= -\infty}^{\infty} a^{n(n+1)/2} b^{n(n-1)/2}, \qquad \qquad |ab| < 1.
\end{align}
Since Ramanujan's theta function has the same degree of generality as the Jacobi theta functions, each classical identity from the theory of theta functions may be stated in terms of $f(a, b)$. 
The Jacobi triple product identity \cite[p. 35]{berndt91} takes the shape
\begin{align} \label{jtp}
  f(a, b) = (-a; ab)_{\infty} (-b; ab)_{\infty} (ab; ab)_{\infty}.
\end{align}
Ramanujan constructed an impressive number of theta function identities by exploiting properties of $f(a, b)$. Among these results is the decomposition theorem from \cite[p. 48]{berndt91}.
\begin{thm} \label{decomp} If $n \in \Bbb Z$, and $U_{n} = a^{n(n+1)/2}b^{n(n-1)/2}$ and $V_{n} = a^{n(n-1)/2}b^{n(n+1)/2}$, then
  \begin{align}
    f(U_{1}, V_{1}) = \sum_{k=0}^{n-1} U_{k} f \left ( \frac{U_{n + k}}{U_{k}}, \frac{V_{n-k}}{U_{k}} \right ).
  \end{align}
\end{thm}
To derive further formulas involving modular forms of level five, we first enumerate identities between the quintic theta functions, the quintic Eisenstein series
\begin{align} \label{gt1}
  E_{1, \chi_{2, 5}}(q) = 1 + (3 - i)\sum_{n=1}^{\infty} \frac{\chi_{2,5}(n) q^{n}}{1 - q^{n}}, \quad E_{1, \chi_{4, 5}}(q) = 1 + (3 + i)\sum_{n=1}^{\infty} \frac{\chi_{4,5}(n) q^{n}}{1 - q^{n}},
\end{align}
and the Rogers-Ramanujan functions, satisfying the Rogers-Ramanujan relations \cite{rogers}
\begin{align*}
G(q) = \sum_{n=0}^{\infty} \frac{q^{n^{2}}}{(q;q)_{n}} = \frac{1}{(q;q^{5})_{\infty} (q^{4};q^{5})_{\infty}}, \quad H(q) = \sum_{n=0}^{\infty} \frac{q^{n^{2}+n}}{(q;q)_{n}} = \frac{1}{(q^{2};q^{5})_{\infty} (q^{3};q^{5})_{\infty}}.
\end{align*}
\begin{thm} \label{nm} Let $\zeta = e^{2 \pi i/5}$. Then
  \begin{align}
\label{vb1} B(q) =  \frac{ f(- q^{2}, -q^{3})}{(q;q)_{\infty}^{3/5}} &= (q;q)^{2/5}  \sum_{n=0}^{\infty} \frac{q^{n^{2}}}{(q;q)_{n}} =  \sqrt[5]{\frac{E_{1, \chi_{4, 5}}(q) + E_{1, \chi_{2, 5}}(q)}{2}},   \\ 
A(q) =   q^{1/5}\frac{f(- q, -q^{4})}{(q;q)_{\infty}^{3/5}} &= q^{1/5}(q;q)^{2/5}  \sum_{n=0}^{\infty} \frac{q^{n^{2}+n}}{(q;q)_{n}}  =   \sqrt[5]{\frac{E_{1, \chi_{4, 5}}(q) - E_{1, \chi_{2, 5}}(q)}{2i}},
\end{align}
\vspace{-0.1in}
\begin{align}
C(q) &=  \frac{f( - \zeta , - \zeta^{4} q)}{(1 - \zeta)(q^{5}; q^{5})_{\infty}^{3/5}} = \sqrt[5]{\frac{(1 + \alpha i) E_{1, \chi_{4, 5}}(q) + (1 - \alpha i)E_{1,\chi_{2, 5}}(q)}{2}}, \\
D(q) &=   \frac{f(- \zeta^{2}, - \zeta^{3} q)}{(1 - \zeta^{2})(q^{5}; q^{5})_{\infty}^{3/5}} = \sqrt[5]{\frac{(1 + \beta i) E_{1, \chi_{4, 5}}(q) + (1 - \beta i)E_{1,\chi_{2, 5}}(q)}{2}}. \label{vb4}
  \end{align}
\end{thm}

Our proof of Theorem \ref{nm} requires several preliminary results relating product expansions for the quintic theta functions and quintic Eisenstein series of weight one.
\begin{lem} \label{j5} Let $\alpha, \beta$ be defined as in Theorem \ref{mainthm}. Then 
\begin{align} \label{dc}
A(q) &=   \frac{q^{1/5}(q;q)_{\infty}^{2/5}}{(q^{2};q^{5})_{\infty}(q^{3};q^{5})_{\infty}}, \quad 
B(q) =  \frac{(q;q)_{\infty}^{2/5}}{(q;q^{5})_{\infty}(q^{4};q^{5})_{\infty}},
\end{align}
\begin{align} \label{bnm}
C(q) &= \frac{(-e^{3 \pi i/5}q; q)_{\infty} (- qe^{- 3 \pi i/5} q; q)_{\infty} (q;q)_{\infty}}{(q^{5}; q^{5})_{\infty}^{3/5}} = (q;q)^{2/5} \prod_{n=1}^{\infty} \frac{(1 + \beta q^{n} + q^{2n} )^{2/5}}{(1 + \alpha q^{n} + q^{2n})^{3/5}} , \\ D(q) &= \frac{(-e^{ \pi i/5}q; q)_{\infty} (- qe^{-  \pi i/5} q; q)_{\infty} (q;q)_{\infty}}{(q^{5}; q^{5})_{\infty}^{3/5}} = (q;q)^{2/5} \prod_{n=1}^{\infty} \frac{(1 + \alpha q^{n} + q^{2n} )^{2/5}}{(1 + \beta q^{n} + q^{2n})^{3/5}}.
\end{align}
\end{lem}
\begin{proof}
To derive the claimed expression for $A(q)$, replace $a$ by $-q$ and $b = -q^{4}$ in the Jacobi triple product identity and multiply both sides by $q^{1/5}(q;q)_{\infty}^{-3/5}$ to obtain
\begin{align*}
   A(q) = q^{1/5} (q;q)_{\infty}^{-3/5} &\sum_{n=-\infty}^{\infty} (-1)^{n} q^{(5n^{2} - 3n)/2}  = q^{1/5} \frac{(q;q^{5})_{\infty}(q^{4};q^{5})_{\infty}(q^{5}; q^{5})_{\infty}}{(q;q)_{\infty}^{3/5}} \\ &= q^{1/5} \frac{(q;q^{5})_{\infty}^{2/5}(q^{4};q^{5})_{\infty}^{2/5}(q^{5};q^{5})_{\infty}^{2/5}}{(q^{2};q^{5})_{\infty}^{3/5}(q^{3};q^{5})_{\infty}^{3/5}} = \frac{q^{1/5}(q;q)_{\infty}^{2/5}}{(q^{2};q^{5})_{\infty}(q^{3};q^{5})_{\infty}}.
\end{align*}
This proves \eqref{dc}. To derive the expansions for $C(q)$, let $a =-e^{3 \pi i/5}q$ and $b = -e^{-3 \pi i/5}q$ in \eqref{jtp}. Multiply both sides of the result by $(1 + e^{3 \pi i/5})^{-1}(q^{5}; q^{5})_{\infty}^{-3/5}$ to obtain
\begin{align*}
  C(q) = &\frac{(q^{5};q^{5})_{\infty}^{-3/5}}{1 + e^{3 \pi i/5}} \sum_{n = - \infty}^{\infty} e^{3 \pi i n/5} q^{(n^{2} - n)/2} =
\frac{(-e^{3 \pi i/5}; q)_{\infty} (- e^{- 3 \pi i/5} q; q)_{\infty} (q;q)_{\infty}}{( 1 + e^{3 \pi i /5})(q^{5}; q^{5})_{\infty}^{3/5}}  \\ &=
\frac{(q;q)_{\infty}}{(q^{5}; q^{5})_{\infty}^{3/5}} \prod_{n=1}^{\infty} (1 + e^{3 \pi i/5}q^{n})(1 + e^{2 \pi i/5}q^{n}) = \frac{(q;q)_{\infty}}{(q^{5}; q^{5})_{\infty}^{3/5}} \prod_{n=1}^{\infty} 1 + \beta q^{n} + q^{2n} .
\end{align*}
The penultimate equality above proves the first identity of \eqref{bnm}. To transcribe the last expression in the form on the right side of \eqref{bnm}, employ the elementary identities
\begin{align}
1 - q^{5n} &= 
\label{b5} (1 + q^{n} + q^{2n} + q^{3n} + q^{4n})(1 - q^{n}), \\
\label{b6} (1 + \alpha q^{n} + q^{2n}&)(1 + \beta q^{n} + q^{2n}) = 1 + q^{n} + q^{2n} + q^{3n} + q^{4n},
\end{align}
to derive
\begin{align*}
C(q) &= (q;q)^{2/5} \prod_{n=1}^{\infty} \frac{(1 + \beta q^{n} + q^{2n} )(1 - q^{n})^{3/5}}{(1 + q^{n} + q^{2n} + q^{3n} + q^{4n})^{3/5}(1 - q^{n})^{3/5}} \\ &= (q;q)^{2/5} \prod_{n=1}^{\infty} \frac{(1 + \beta q^{n} + q^{2n} )}{(1 + \alpha q^{n} + q^{2n})^{3/5} (1 + \beta q^{n} + q^{2n})^{3/5}}  \\ &= (q;q)^{2/5} \prod_{n=1}^{\infty} \frac{(1 + \beta q^{n} + q^{2n} )^{2/5}}{(1 + \alpha q^{n} + q^{2n})^{3/5}}.
\end{align*}
We have therefore deduced the claimed expansions for $A(q)$ and $C(q)$. The displayed product expansions for $B(q)$ and $D(q)$ are derived in a similar fashion.
\end{proof}

 \begin{lem} \label{th1} Define the series $e_{\alpha}(q)$ as in \eqref{ser}, and $\alpha, \beta$ as in Theorem \ref{mainthm}. Then 
      \begin{align} \label{h1}
    \frac{e_{2/5}(q) - e_{1/5}(q)}{2 \sqrt{5}} &= q \frac{(q;q^{5})_{\infty}^{2}(q^{4};q^{5})_{\infty}^{2}(q^{5};q^{5})_{\infty}^{2}}{(q^{2};q^{5})_{\infty}^{3}(q^{3};q^{5})_{\infty}^{3}},
    \end{align}
\begin{align} \label{h2}
\frac{\alpha^{3} e_{1/5}(q) - \beta^{3}e_{2/5}(q)}{2 \sqrt{5}} &= \frac{(q^{2};q^{5})_{\infty}^{2}(q^{3};q^{5})_{\infty}^{2}(q^{5};q^{5})_{\infty}^{2}}{(q;q^{5})_{\infty}^{3}(q^{4};q^{5})_{\infty}^{3}},
  \end{align}
 \begin{align} \label{fb}
\frac{3 \alpha e_{2/5}(q) - \alpha^{4} e_{1/5}(q)}{2} = (q;q)_{\infty}^{2} \prod_{n=1}^{\infty} \frac{(1 + \beta q^{n} + q^{2n})^{2}}{(1 + \alpha q^{n} + q^{2n})^{3}},
  \end{align}
  \begin{align} \label{fb1}
\frac{3 \beta e_{1/5}(q) - \beta^{4} e_{2/5}(q)}{2} = (q;q)_{\infty}^{2} \prod_{n=1}^{\infty} \frac{(1 + \alpha q^{n} + q^{2n})^{2}}{(1 + \beta q^{n} + q^{2n})^{3}}.
  \end{align}
  \end{lem}

Identities \eqref{fb}-\eqref{fb1} are shown in \cite{huber_quint} to be special cases of the classical identity 
\begin{align} \label{zs}
  \zeta(\theta - a) - \zeta(\theta - b) - \zeta(a - b) + \zeta(2 a - 2 b) = \frac{\sigma(\theta - 2 a + b) \sigma(\theta - 2 b + a)}{2\sigma(2 b - 2 a)\sigma(\theta - a) \sigma(\theta - b)},
\end{align} 
relating the Weierstrass $\zeta$-function to the Weierstrass $\sigma$-function
\begin{align} \label{sigma}
  \sigma(\theta) =  2\exp \left ( \frac{E_{2}(q)}{24} \theta^{2} \right ) \sin \left ( \frac{\theta}{2} \right ) \prod_{n=1}^{\infty} \frac{1 - 2 q^{n} \cos(\theta) + q^{2n}}{(1 - q^{n})^{2}}.
\end{align} 
The first two identities of Lemma \ref{th1} are derived in \cite{qeis} by identifying the logarithmic derivatives of each side of the proposed relation through the use of Theorem \ref{coupl} and appropriate elliptic function identities. In the following lemma, the elliptic parameters of Lemma \ref{th1} are transcribed in terms of quintic Eisenstein series of weight one defined by \eqref{gt1}.  For a complete proof of Lemmas  \ref{j5}--\ref{theth}, the reader is referred to \cite{qeis}.
\begin{lem} \label{theth}
If the characters $\chi_{2,5}$ and $\chi_{4,5}$ are defined as in \eqref{c1}--\eqref{c2}, then 
\begin{align}
  e_{1/5}(q) \label{tw1}
&=  \frac{E_{1, \chi_{4, 5}}(q) +  E_{1, \chi_{2, 5}}(q)}{2} +  \beta^{3}\frac{E_{1, \chi_{4, 5}}(q) - E_{1, \chi_{2, 5}}(q)}{2i}, \qquad \beta = \frac{1 - \sqrt{5}}{2},\\ 
  e_{2/5}(q) &=  \frac{E_{1, \chi_{4, 5}}(q) +  E_{1, \chi_{2, 5}}(q)}{2} +  \alpha^{3}\frac{E_{1, \chi_{4, 5}}(q) - E_{1, \chi_{2, 5}}(q)}{2i},  \qquad \alpha = \frac{1 + \sqrt{5}}{2}. \label{tw2}
\end{align}
\end{lem}
\begin{proof}[Proof of Theorem \ref{nm}]
  The former equalities of \eqref{vb1}--\eqref{vb4} follow from \eqref{ramt} and the definitions of the quintic theta functions from Section \ref{s1}. The latter equalities of Theorem \ref{nm} follow from the Rogers-Ramanujan relations and Lemmas \ref{j5}--\ref{theth}.
\end{proof}
\begin{thm} \label{jeq}
  The identities of Theorem \ref{mainthm} hold, and are equivalent to \eqref{gj6}--\eqref{po2}.
\end{thm}
\begin{proof}
The first two equalities on line \eqref{d00} of Theorem \ref{mainthm} follow immediately from the latter equalities of Theorem \ref{nm} relating theta functions and Eisenstein series. The linear relations appearing on line \eqref{po} of Theorem \ref{mainthm} may be deduced from Theorems \ref{nm} and \ref{decomp} via two identities Ramanujan's Lost Notebook \cite[pp. 22]{andrews05}, with $ \zeta = e^{2 \pi i/5}$, 
\begin{align}
  f(- q^{2}, - q^{3}) - \alpha q^{1/5} f(-q, -q^{4}) &= f(- \zeta^{2}, -\zeta^{3} q^{1/5})/(1 - \zeta^{2}), \\ f(- q^{2}, - q^{3}) - \beta q^{1/5} f(-q, -q^{4}) &= f(- \zeta, -\zeta^{4} q^{1/5})/(1 - \zeta).
\end{align}
To prove the equivalence of Ramanujan's formulas \eqref{gj6}--\eqref{po2} and Theorem \ref{mainthm}, we use the Jacobi triple product identity \eqref{jtp} and the Rogers-Ramanujan identities to transcribe the first relation on line \eqref{d00} as
\begin{align}
 \label{n6} G^{5}(q) - \alpha^{5} H^{5}(q) &= \frac{(q;q)_{\infty}^{3}}{(q^{5}; q^{5})_{\infty}^{3}} \prod_{n=1}^{\infty} (1 + \beta q^{n} + q^{2n})^{5}.
\end{align}
Divide both sides of \eqref{n6} by $G^{5/2}(q) H^{5/2}(q)$
 and apply \eqref{b5}--\eqref{b6} to obtain
\begin{align}
  R^{-5/2}(q) - \alpha^{5} R^{5/2}(q) &=  \frac{(q;q)_{\infty}^{11/2}}{(q^{5}; q^{5})_{\infty}^{11/2}}  \prod_{n=1}^{\infty} \frac{(1 + q^{n} + q^{2n} + q^{3n } q^{4n})^{5}}{(1 + \alpha q^{n} + q^{2n})^{5}} \\ 
&= \frac{(q^{5}; q^{5})_{\infty}^{5}}{(q;q)_{\infty}^{5}} \frac{(q;q)_{\infty}^{11/2}}{(q^{5}; q^{5})_{\infty}^{11/2}}  \prod_{n=1}^{\infty} \frac{1}{(1 + \alpha q^{n} + q^{2n})^{5}}
\\ &= \sqrt{\frac{(q;q)_{\infty}}{(q^{5}; q^{5})_{\infty}}} \prod_{n=1}^{\infty} \frac{1}{(1 + \alpha q^{n} + q^{2n})^{5}}.
\end{align}
This proves that the left equation on line \eqref{d00} is equivalent to \eqref{po1}. The remaining identities of Theorem \ref{mainthm} may be reconciled with \eqref{gj6}--\eqref{po2} in a similar way.
\end{proof}
The identities of Theorem \ref{mainthm} may be obtained from each other by applying modular transformation formulas. We prove the relevant relations in Section \ref{modular}.

\section{The Principles of Pentication and Pentamidiation} \label{s3}
Jacobi's Principles of Duplication and Dimidiation \cite{funnov} were rediscovered and employed by Ramanujan \cite[p. 121]{spiritraman} to give parameterizations for theta functions and Eisenstein series of argument $\pm q^{2^{n}}, n \in \Bbb Z$. These parameterizations are based on Ramanujan's equivalent formulation of the Jacobi inversion formula \cite[Ch. 17]{ramnote}
\begin{align} \label{jacinv}
\theta_{3}(q) &= \sqrt{{_{2}F}_{1}(1/2, 1/2; 1; x)} := \sqrt{z}, \quad q = \exp \left ( - \pi \frac{{_{2}F}_{1}(1/2, 1/2; 1; 1-x)}{{_{2}F}_{1}(1/2, 1/2; 1; x)} \right ),
\end{align}
where  $\theta_{2}(q)$ is defined as in \eqref{jac0} and ${_{2}F}_{1}(a, b; c ; z)$ denotes the hypergeometric function \cite[p. 64]{aaa1}. Ramanujan employed \eqref{jacinv} to derive identities of the form
\begin{align}
 \theta_{3}(q^{2}) &= \sqrt{z} \sqrt{ \frac{1}{2}(1 + \sqrt{1 - x})}, \qquad  \theta_{3}(q^{1/2}) = \sqrt{z}(1 + \sqrt{x})^{1/2},  \\ \theta_{3}(-q^{2}) &= \sqrt{z}(1 - x)^{1/8}, 
\qquad  \theta_{3}(q^{4}) = \frac{1}{2} \sqrt{z} (1 + (1-x)^{1/4}). \label{fx}
\end{align}

\begin{thm}
  Suppose $q, x, z$ satisfy an equation of the form $\Omega(q, x, z)$. Then the following Duplication and Dimidiation formulas hold:
  \begin{align}
  &\Omega \left ( q^{2}, \left ( \frac{1 - \sqrt{1 - x}}{1 + \sqrt{1 - x}} \right )^{2}, \frac{1}{2} (1 + \sqrt{1 - x}) z \right ) = 0, \\ 
& \qquad \ \Omega \left ( q^{1/2}, \frac{4 \sqrt{x}}{(1 + \sqrt{x})^{2}}, (1 + \sqrt{x}) z \right ) = 0.
  \end{align}
\end{thm}
Formulas analogous to \eqref{jacinv}-\eqref{fx} involving the cubic theta functions \eqref{abq}--\eqref{cq} were also derived by Ramanujan \cite{alt_base}, including the inversion formula 
\begin{align} \label{cubinv}
a(q) &= {_{2}F}_{1}(1/3, 2/3; 1; x) := z, \quad q = \exp \left ( - \frac{2\pi}{\sqrt{3}} \frac{{_{2}F}_{1}(1/3, 2/3; 1; 1-x)}{{_{2}F}_{1}(1/3, 2/3; 1; x)} \right ).
\end{align}
\begin{thm}
  Let $x, z, q$ be related by \eqref{cubinv}. 
If $q, x, z$ are related by the equation $\Omega(q, x, z) = 0$, then the following triplication and trimidiation formulas hold:
  \begin{align}
    \Omega \left ( q^{3}, \left \{ \frac{1 - (1 - t^{3})^{1/3}}{1 + 2(1 - t^{3})^{1/3}} \right \}^{3}, \frac{1}{3} \left \{ 1 + 2(1 - t^{3})^{1/3} \right \} z \right ) = 0,
  \end{align}
\begin{align}
  \Omega \left ( q^{1/3}, 1 - \left ( \frac{1 - t}{1 + 2 t} \right )^{3}, (1 + 2t)z \right ) = 0.
\end{align}
\end{thm}

Inversion formulas of level five are not as concisely expressible in terms of hypergeometric functions. However, analogous formulas may be constructed from the work of F. Beukers \cite{MR702189} and T. Mano \cite{MR1932736} by defining $t = R^{5}(q)$, where $R(q) = A(q)/B(q)$ denotes the Rogers-Ramanujan continued fraction. The inverse function $q(t)$ then arises as a quotient of suitably chosen linearly independent solutions to the Picard-Fuchs equation
\begin{align}
  t(t^{2} + 11 t -1) \frac{d^{2} f(t)}{dt^{2}} + (3t^{2} + 22t -1) \frac{df(t)}{dt} + (t + 3) f(t) = 0.
\end{align}

In what follows, we demonstrate a systematic procedure for computing quintic parameterizations of theta functions and Eisenstein series with argument $\zeta^{k} q^{5^{n}}$, $n,k \in \Bbb Z$, $\zeta = e^{2 \pi i/5}$, from initial quintic parameterizations of argument $q$. 
\begin{thm} \label{quint}
Let $A = A(q)$ and $B = B(q)$. Suppose that a relation of the form $\Omega(q, A, B ) = 0$ holds. Then we have the pentamidiation formula $\Omega \left ( q^{1/5}, \mathcal{A}, \mathcal{B} \right ) = 0$
 and the corresponding  pentication formula $\Omega \left ( q^{5}, \mathcal{C}, \mathcal{D} \right ) = 0$, where $\mathcal{A, B, C, D}$ are defined by
\begin{align}
\mathcal{A} &= \sqrt[5]{A^{5}-3 A^4 B+4 A^{3} B^{2}-2 A^{2} B^{3}+A B^{4}}, \\ 
\mathcal{B} &= \sqrt[5]{B^{5}+3 B^4 A+4 B^{3} A^{2}+2 B^{2} A^{3}+B A^{4}},
\end{align}
\begin{align}
\mathcal{C} = \frac{\sqrt[5]{B^{5} - \alpha^{5} A^{5}} - \sqrt[5]{B^{5} - \beta^{5} A^{5}}}{\beta - \alpha}, \quad 
\mathcal{D} = \frac{\beta \sqrt[5]{B^{5} - \alpha^{5} A^{5}} - \alpha \sqrt[5]{B^{5} - \beta^{5} A^{5}}}{\beta - \alpha}.
\end{align}
\end{thm}
\begin{proof}
From Theorem \ref{mainthm}, we derive
\begin{align*}
  B(q^{1/5}) &= \sqrt[5]{ \frac{\alpha^{5} D^{5}(q^{1/5}) - \beta^{5} C^{5}(q^{1/5})}{\alpha^{5} - \beta^{5}}} \\ &= \sqrt[5]{ \frac{\alpha^{5} (B(q) - \beta A(q))^{5} - \beta^{5} (B(q) - \alpha A(q))^{5}}{\alpha^{5} - \beta^{5}}} \\  &= \sqrt[5]{B^{5}(q)+3 B^4(q) A(q)+4 B^{3}(q) A^{2}(q)+2 B^{2}(q) A^{3}(q)+B(q) A^{4}(q)}.
\end{align*}
We may show similarly from Theorem \ref{mainthm} that $\mathcal{A}(q) = A(q^{1/5})$, $\mathcal{C}(q) = A(q^{5})$, and $\mathcal{D}(q) = B(q^{5})$. Here $z^{1/5}$ is the branch of the fifth root satisfying $\lim_{z \to 1} z^{1/5} = 1$.
\end{proof}
By iterating Theorem \ref{quint}, a number of unwieldy representations for higher-order transformed quintic theta functions may be given in terms of nested radicals. In each of the following formulas, the principal branch of the fifth root is indicated.
\begin{cor} \label{asv} If $\alpha$ and $\beta$ are defined as in Theorem \ref{mainthm}, then 
  \begin{align*}
     A(q^{1/25}) =  &\left ( 
 \frac{\left(\sqrt[5]{-\frac{\beta ^5 (B(q)-\alpha  A(q))^5-\alpha ^5 (B(q)-\beta  A(q))^5}{5 \sqrt{5}}}-\alpha  \sqrt[5]{-\frac{(B(q)-\alpha  A(q))^5-(B(q)-\beta  A(q))^5}{5 \sqrt{5}}}\right)^5}{5 \sqrt{5}}  \right. \\ & \quad \left. -\frac{\left(\beta  \sqrt[5]{-\frac{(B(q)-\alpha  A(q))^5-(B(q)-\beta  A(q))^5}{5 \sqrt{5}}}-\sqrt[5]{-\frac{\beta ^5 (B(q)-\alpha  A(q))^5-\alpha ^5 (B(q)-\beta  A(q))^5}{5\sqrt{5}}}\right)^5}{5 \sqrt{5}} \right )^{1/5}
  \end{align*}
 \begin{align*}
     B(q^{1/25}) =  &\left ( \frac{\beta^{5}\left(\sqrt[5]{-\frac{\beta ^5 (B(q)-\alpha  A(q))^5-\alpha ^5 (B(q)-\beta  A(q))^5}{5 \sqrt{5}}}-\alpha  \sqrt[5]{-\frac{(B(q)-\alpha  A(q))^5-(B(q)-\beta  A(q))^5}{5 \sqrt{5}}}\right)^5}{5 \sqrt{5}}  \right. \\ & \quad \left. -\frac{ \alpha^{5}\left(\beta  \sqrt[5]{-\frac{(B(q)-\alpha  A(q))^5-(B(q)-\beta  A(q))^5}{5 \sqrt{5}}}-\sqrt[5]{-\frac{\beta ^5 (B(q)-\alpha  A(q))^5-\alpha ^5 (B(q)-\beta  A(q))^5}{5\sqrt{5}}}\right)^5}{5 \sqrt{5}} \right )^{1/5}
  \end{align*}
  \begin{align}
    A(q^{25}) = &\left ( \frac{\sqrt[5]{\frac{\beta  \left(\beta ^4-1\right)  \sqrt[5]{B^{5}(q) - \alpha^{5} A^{5}(q)} +   \left(\alpha -\beta ^5\right) \sqrt[5]{B^{5}(q) - \beta^{5} A^{5}(q)}}{\alpha -\beta }}}{\alpha - \beta}  \right. \\ & \qquad \qquad - \left. 
\frac{\sqrt[5]{\frac{  \left(\alpha ^5-\beta \right) \sqrt[5]{B^{5}(q) - \alpha^{5} A^{5}(q)}-\alpha  \left(\alpha ^4-1\right)  \sqrt[5]{B^{5}(q) - \beta^{5} A^{5}(q)}}{\alpha -\beta }}}{\alpha -\beta } \right )^{1/5},  \nonumber \\
    B(q^{25}) &= \left ( \frac{\alpha  \sqrt[5]{\frac{\beta  \left(\beta ^4-1\right) \sqrt[5]{B^{5}(q) - \alpha^{5} A^{5}(q)} + \left(\alpha -\beta ^5\right)\sqrt[5]{B^{5}(q) - \beta^{5} A^{5}(q)}}{\alpha -\beta }}}{\alpha - \beta} \right.  \\ &  \left.  \frac{-\beta  \sqrt[5]{\frac{ \left(\alpha ^5-\beta \right)\sqrt[5]{B^{5}(q) - \alpha^{5} A^{5}(q)} -\alpha  \left(\alpha ^4-1\right)  \sqrt[5]{B^{5}(q) - \beta^{5} A^{5}(q)}}{\alpha -\beta }}}{\alpha -\beta } \right )^{1/5}. \nonumber 
  \end{align}
\end{cor}
The authors of \cite{qeis} observed that a variety of important modular parameters may be expressed as homogeneous polynomials in the functions $A^{5}(q)$ and $B^{5}(q)$. 
We describe how these series transform when $q$ is replaced by $q^{1/5}$ via pentamidiation arrays whose entries are sums of multinomial coefficients.
\begin{defn} \label{flb}
Let $\mathcal{F}_{n}$ denote the set of ordered nonnegative integer partitions of $n$ containing exactly five parts, $\vec{n} = (n_{1}, n_{2}, n_{3}, n_{4}, n_{5})$, and let $f, g: \Bbb Z^{5} \to \Bbb Z$ denote the functions
\begin{align}
f(\vec{n}) = 5n_{1} + 4n_{2} + 3n_{3} + 2n_{4} + n_{5}, \quad  g(\vec{n}) =  n_{2} + 2n_{3} + 3n_{4} + 4n_{5}.
\end{align}
Define the pentamidiation array $\mathcal{B}_{d}$ to be the $(5d +1) \times (d +1)$ array with $(r, k)$th entry
\begin{align}
\Bigl ( \mathcal{B}_{d} \Bigr )_{r,k} = \sum_{\substack{\vec{j} \in \mathcal{F}_{d-k} \vec{k} \in \mathcal{F}_{k} \\ f(\vec{j}) + g(\vec{k}) = 5 d-r \\ g(\vec{j}) + f(\vec{k}) = r }} (-1)^{k_{2} + k_{4}} 2^{j_{4} + k_{4}} 3^{j_{2} + k_{2}} 4^{j_{3} + k_{3}} {k \choose \vec{k}} {d-k \choose \vec{j}}.
\end{align}
\end{defn} 
The first two pentamidiation arrays are
\begin{align} \label{fpent}
\mathcal{B}_{1} =  \begin{pmatrix}
    1 & 3 & 4 & 2 & 1 & 0 \\ 0 & 1 & -2 & 4 & -3 & 1
  \end{pmatrix}^{T},
\end{align}
\begin{align}
\mathcal{B}_{2} =
\left ( \begin{array}{ccccccccccc}
  1 & 6 & 17 & 28 & 30 & 22 & 12 & 4 & 1 & 0 & 0 \\ 
0 & 1 & 1 & 2 & 3 & 5 & -3 & 2 & -1 & 1 & 0 \\ 
0 & 0 & 1 & -4 & 12 & -22 & 30 & -28 & 17 & -6 & 1
\end{array} \right )^{T}. 
\end{align}

\begin{lem} \label{refin1}
Let $\mathcal{B}_{d}$ be defined as in Definition \ref{flb}. The claims of Theorem \ref{gt} hold. 
\end{lem}
\begin{proof}
  Apply Theorem \ref{quint} and the Multinomial Theorem to deduce
  \begin{align} \label{bgty}
    \sum_{k=0}^{d} &a_{k}  A^{5k}(q^{1/5}) B^{5(d - k)}(q^{1/5}) =  \sum_{k=0}^{d} a_{k}  \mathcal{A}^{5k} \mathcal{B}^{5(d - k)} \\ = \sum_{k=0}^{d} & a_k\Bigl \{ \Bigl (A^{5}(q)-3 A^4(q) B(q)+4 A^{3}(q) B^{2}(q)-2 A^{2}(q) B^{3}(q)+A(q) B^{4}(q) \Bigr )^{k} \nonumber \\ & \quad \times  \Bigl (B^{5}(q)+3 B^4(q) A(q)+4 B^{3}(q) A^{2}(q)+2 B^{2}(q) A^{3}(q)+B(q) A^{4}(q) \Bigr ))^{d - k} \Bigr \}  \nonumber \\ &= \sum_{k=0}^{d} a_{k} \sum_{\substack{\vec{j} \in \mathcal{F}_{d-k} \\ \vec{k} \in \mathcal{F}_{k} }} (-1)^{k_{2} + k_{4}} 2^{j_{4} + k_{4}} 3^{j_{2} + k_{2}} 4^{j_{3} + k_{3}} {k \choose \vec{k}} {d-k \choose \vec{j}} A^{g(\vec{j}) + f(\vec{k})}(q) B^{f(\vec{j}) + g(\vec{k})}(q) \nonumber \\ &= \sum_{\substack{0 \le k \le d \\ 0 \le r \le 5d }} a_{k} \sum_{\substack{\vec{j} \in \mathcal{F}_{d-k},\  \vec{k} \in \mathcal{F}_{k} \\ f(\vec{j}) + g(\vec{k}) = 5 d-r \\ g(\vec{j}) + f(\vec{k}) = r }} (-1)^{k_{2} + k_{4}} 2^{j_{4} + k_{4}} 3^{j_{2} + k_{2}} 4^{j_{3} + k_{3}} {k \choose \vec{k}} {d-k \choose \vec{j}} A^{r}(q) B^{5d -r}(q). \nonumber 
  \end{align}
Therefore, referring to Definition \ref{flb}, the $(r, k)$th entry of the pentamidiation array $\mathcal{B}_{d}$ coincides with the coefficient of $a_{k} A^{r}(q)B^{5 d -r}$ in the multinomial expansion of \eqref{bgty}. It follows that the $r$th entry of the product $\mathcal{B}_{d}(a_{0}, a_{1}, \ldots, a_{d})^{T} = (b_{0}, b_{1}, \ldots, b_{5 d})^{T}$ represents the coefficient of $A^{r}(q)B^{5 d-r}(q)$ in the expansion of \eqref{bgty}. This proves \eqref{cdq}. To prove \eqref{cdq1}, note that, for each reduced residue $r$ modulo five, the contribution to the coefficient of $q^{(5 k + r)/5}, k \in \Bbb N$, in the $q$-expansion of the series from \eqref{bgty} comes entirely from the terms $A^{n}(q)B^{5 d -n}$, $n \equiv r \pmod{5}$, in the final equality of \eqref{bgty}.
\end{proof}

Of particular interest in the ensuing examples are the $(d+1) \times (d +1)$ matrices associated with the operator $\Omega_{5, 0}$ appearing in Theorem \ref{gt}. For each homogeneous polynomial parameterization for a function $f(q)$ in terms of $A^{5}(q)$ and $B^{5}(q)$, the $m$-fold iteration of $\Omega_{5,0}$ on $f(q)$ may be expressed as a polynomial of the same degree with coefficients expressible through an appropriate power of a generating matrix.
\begin{cor} \label{tipp}
Define $\mathcal{B}_{d}$ as in Definition \ref{flb}. Let $\mathcal{A}_{d}$ be the $(d+1) \times (d+1)$ matrix obtained from $\mathcal{B}_{d}$ by labeling the rows of $\mathcal{B}_{d}$ from $0$ to $5d$ and removing the rows corresponding to nonzero residue classes modulo five. Then, for any homogeneous polynomial $F(q)$ of degree $d$ in $A^{5}(q)$ and $B^{5}(q)$, with
\begin{align*}
  F(q) = \sum_{k=0}^{d} a_{k}A^{5k}(q) B^{5(d - k)}(q) = \sum_{n=0}^{\infty} f(n) q^{n},
\end{align*}
the $m$-fold quintic multisection of $F(q)$ has the representation 
\begin{align*}
   \sum_{n=0}^{\infty} f(5^{m}n) q^{n} = \sum_{k=0}^{d} b_{k,m}A^{5k}(q) B^{5(d - k)}(q), \qquad m \in \Bbb N,
\end{align*}
where $(b_{0,m}, b_{1,m}, \ldots, b_{d,m})^{T} = \mathcal{A}_{d}^{m}(a_{0}, a_{1}, \ldots, a_{d})^{T}$.  
\end{cor}
The first six matrices $\mathcal{A}_{d}$ are as follows:
{\allowdisplaybreaks \begin{align*}
  \mathcal{A}_{1} &= I_{2}, \qquad \mathcal{A}_{2} =
\left(
\begin{array}{ccc}
 1 & 0 & 0 \\
 22 & 5 & -22 \\
 0 & 0 & 1
\end{array}
\right),
 \qquad \mathcal{A}_{3} = \left(
\begin{array}{cccc}
 1 & 0 & 0 & 0 \\
 264 & 25 & 0 & 24 \\
 24 & 0 & 25 & -264 \\
 0 & 0 & 0 & 1
\end{array}
\right),
\end{align*}
\begin{align*}
\mathcal{A}_{4} &= \left(
\begin{array}{ccccc}
 1 & 0 & 0 & 0 & 0 \\
 1356 & 115 & 10 & 10 & -8 \\
 1462 & 110 & 15 & -110 & 1462 \\
 8 & 10 & -10 & 115 & -1356 \\
 0 & 0 & 0 & 0 & 1
\end{array}
\right), 
\end{align*}
\begin{align*}
\mathcal{A}_{5} &= 
\left(
\begin{array}{cccccc}
 1 & 0 & 0 & 0 & 0 & 0 \\
 4603 & 410 & 35 & 5 & -5 & 1 \\
 25494 & 2360 & 235 & -20 & 270 & -2272 \\
 2272 & 270 & 20 & 235 & -2360 & 25494 \\
 1 & 5 & 5 & -35 & 410 & -4603 \\
 0 & 0 & 0 & 0 & 0 & 1
\end{array}
\right),  
\end{align*}
\begin{align*}
\mathcal{A}_{6} &= \left(
\begin{array}{ccccccc}
 1 & 0 & 0 & 0 & 0 & 0 & 0 \\
 12228 & 1126 & 102 & 9 & -2 & 1 & 0 \\
 232494 & 21353 & 1931 & 177 & 94 & -647 & 1626 \\
 108772 & 8994 & 688 & 71 & -688 & 8994 & -108772 \\
 1626 & 647 & 94 & -177 & 1931 & -21353 & 232494 \\
 0 & 1 & 2 & 9 & -102 & 1126 & -12228 \\
 0 & 0 & 0 & 0 & 0 & 0 & 1
\end{array}
\right). 
\end{align*}
The role of the matrices $\mathcal{A}_{d}$ in the quintic multisection of modular forms will be further explored in the next few sections. 
Because of their utility in constructing parameterizations for quintic multisections and congruence relations for the coefficients of modular forms, the eigenstructure of $\mathcal{A}_{d}$ should be given an independent analysis. We make an initial observation about the determinants of the matrices $\mathcal{A}_{d}$ that should be refined. 
\begin{conj} Let $\mathcal{A}_{d}$ be defined as in Corollary \ref{tipp}.
The determinant of $\mathcal{A}_{d}$ is equal to $\pm 5^{f(d)}$ for some $f(d) \in \Bbb N$. 
\end{conj}

We end this section with a quintic counterpart of Jacobi's change of sign formula \cite[p. 126]{berndt91}. Corollary \ref{yu} follows directly from Theorem \ref{mainthm} and may be used in conjunction with Theorem \ref{quint} to express series of argument $\xi^{k} q^{5^{n}}, n \in \Bbb Z$ in terms of $A(q)$ and $B(q)$.
\begin{cor} \label{yu} 
  Let $\xi = e^{2 \pi i/5}$. Suppose that a relation of the form $\Omega(q, A(q), B(q)) = 0$ holds. Then we also have, for each $k = 0, 1, 2, 3, 4$, 
  \begin{align}
\Omega \left ( \zeta^{k} q, \sqrt[5]{ \frac{C^{5}(\xi^{k} q) - D^{5}(\xi^{k} q)}{\beta^{5} - \alpha^{5}}}, \sqrt[5]{ \frac{ \beta^{5} C^{5}(\xi^{k} q) - \alpha^{5} D^{5}(\xi^{k} q)}{\beta^{5} - \alpha^{5}}} \right ).
  \end{align}
where
\begin{align}
  C(\xi^{k} q) = B(q^{5}) - \alpha \xi^{k} A(q^{5}), \qquad   &\hbox{and} \quad D(\xi^{k} q) =  B(q^{5}) - \beta \xi^{k} A(q^{5}).
\end{align}
\end{cor}

The relevance of Corollary \ref{yu} in computing quintic multisections arises from the Simpson-Waring dissection formula \cite{simpson} \cite[p. xlii]{MR1146921}
\begin{align}
  f(x) = \sum_{n=0}^{\infty} a_{n}x^{n}, \qquad \sum_{n=0}^{\infty} a_{kn+m} x^{kn+m} = \frac{1}{k} \sum_{j=0}^{k-1} \omega^{-jm} f(\omega^{j} x), \quad \omega = e^{2 \pi i/k}.
\end{align}

\section{Quintic Decomposition of Eisenstein series} \label{s4}
We now consider applications of the previous section to Eisenstein series, including those defined by \eqref{eisdef} and normalizations of their corresponding Fourier expansions at zero 
\begin{align} \label{gp}
  L_{k, \chi}(q) = \sum_{n=1}^{\infty} \frac{n^{k-1}}{1 - q^{5n}} \sum_{m=1}^{4} \chi(m) q^{mn}, \qquad k \ge 2, \quad \chi(-1) = (-1)^{k}.
\end{align}
These series are among the classes of Eisenstein series considered in \cite{qeis} that are expressed in terms of symmetric homogeneous polynomials in $A^{5}(q)$ and $B^{5}(q)$.
\begin{thm} \label{hda} If $A = A(q)$, $B = B(q)$, and $E_{2}(q)$ is defined as in Theorem \ref{d_quint}, then
  \begin{align}
    L_{2, \chi_{3,5}}(q) &= A^{5}B^{5}, \qquad L_{4, \chi_{3,5}}(q) = B^{15}A^{5}+BA^{15}, \label{sp} \\ L_{6, \chi_{3, 5}}(q) &= B^{25}A^{5} + 18 B^{20}A^{10} + 14A^{15}B^{15} - 18B^{10}A^{20} +B^{5}A^{25}, \\ L_{2, \chi_{1,5}}(q) &= \frac{A^{10} + B^{10} - E_{2}(q^{5})}{6}, \qquad L_{4, \chi_{1,5}}(q) = B^{15}A^{5} + 2 B^{10}A^{10} - B^{5}A^{15}, \\ L_{6, \chi_{1,5}}(q) &= B^{25}A^{5} + 20 B^{20}A^{10} + 20 B^{10}A^{20} - B^{5}A^{20}.
  \end{align}
\end{thm}
Corresponding parameterizations exist for the Eisenstein series on the full modular group, defined by \eqref{eis_full}.
The following formulas for $E_{4}(q)$ and $E_{6}(q)$ provide a point of departure for further identities. Formulas \eqref{hjk1}--\eqref{hjk2} will follow from Theorem \ref{d_quint}. We defer their proof until Section \ref{s6}. Identities \eqref{g1}--\eqref{g2} are equivalent to identities \cite[pp. 50-51]{ramlost} derived by Ramanujan and are proven in \cite{MR1441328,qeis}.
\begin{thm} \label{nbr} Let $A = A(q)$ and $B = B(q)$.
  \begin{align} 
E_{2}(q) &= A^{10} + 66 A^{5} B^{5} - 11 B^{10} + 60 q \frac{d}{dq} \log A   \label{hjk1} \\ 
&= B^{10} - 66 A^{5} B^{5} - 11 A^{10} + 60 q \frac{d}{dq} \log B, \label{hjk2} 
\end{align}
\begin{align}
 \label{g1}
  E_{4}(q) &
= B^{20} + 228 B^{15}A^{5} + 494 B^{10}A^{10} - 228 B^{5}A^{15} + A^{20}, \\
 \label{g2}
    E_{6}(q) &
= B^{30}- 522 B^{25}A^{5} - 10005 B^{20}A^{10}- 10005B^{10}A^{20} + 522B^{5}A^{25}+ A^{30}.
\end{align}
\end{thm}
Along with the previous formulas, we will make use of corresponding parameterizations from \cite{qeis} for the quintic Eisenstein series in terms of the quintic theta functions. 
\begin{thm} \label{par1} Let $A = A(q)$ and $B = B(q)$.  
  \begin{align} \label{vfe}
   E_{1, \chi_{4, 5}}(q) &= B^{5} + i A^{5}, \quad  \quad \ E_{1, \chi_{2, 5}}(q) = B^{5} - i A^{5}, \\   \label{vfe1} E_{2, \chi_{1, 5}}(q) &= A^{10} + B^{10}, \quad \ \ \ E_{2, \chi_{3,5}}(q) = B^{10} - 11 A^{5} B^{5} - A^{10}, \\  \label{vfe2} E_{3, \chi_{2, 5}}(q) &= E_{1, \chi_{2, 5}}(q) E_{2, \chi_{3, 5}}(q), \quad E_{3, \chi_{4, 5}}(q) = E_{1, \chi_{4, 5}}(q) E_{2, \chi_{3, 5}}(q), 
\end{align}
\end{thm}
Equations \eqref{vfe} and the Principle of Pentamidiation together imply Theorem \ref{fjl}.
\begin{cor} \label{pr5} If $\chi_{4, 5}$ denotes the Dirichlet character $\langle \chi_{4,5}(n) \rangle_{n=0}^{4} = \langle 0, 1, -i, i, -1 \rangle$,  
  \begin{align*}
  q^{4/5} \sum_{n=0}^{\infty} \Bigl ( \sum_{d \mid 5 n +4} \chi_{4, 5}(d) \Bigr )q^{n} &= -i A^{4}B, \qquad  
  q^{3/5} \sum_{n=0}^{\infty} \Bigl ( \sum_{d \mid 5 n +3} \chi_{4, 5}(d) \Bigr )q^{n} = (1 + i) A^{3}B^{2}, \\ 
  q^{2/5} \sum_{n=0}^{\infty} \Bigl ( \sum_{d \mid 5 n +2} \chi_{4, 5}(d) \Bigr )q^{n} &= (1 - i) A^{2}B^{3}, \qquad  
  q^{1/5} \sum_{n=0}^{\infty} \Bigl ( \sum_{d \mid 5 n +1} \chi_{4, 5}(d) \Bigr )q^{n} =  AB^{4}.
  \end{align*}
\end{cor}
\begin{proof}
  Apply the pentamidiation array $\mathcal{B}_{1}$ to the parameterization from \eqref{vfe} to yield
  \begin{align} \label{mbr}
    E_{1, \chi_{4, 5}}(q^{1/5}) &= 1 + (3 + i) \sum_{n=1}^{\infty} \frac{\chi_{4,5}(n) q^{n/5}}{1 - q^{n/5}} \\ &= B^{5} + (3 +i) B^{4}A + (4 - 2i) B^{3}A^{2} + (2 +4i) B^{2}A^{3} + (1 - 3i)B A^{4} + iA^{5}. \nonumber
  \end{align}
The claimed identities follow by equating terms of the form $q^{k/5}$ for each nonzero residue class $k$ modulo five. In particular, the only contribution to such terms on the right hand side of \eqref{mbr} come from $A^{r}B^{5-r}$, for values $r \equiv k \pmod{5}$.
\end{proof}
We likewise decompose the Eisenstein series of weight two for the full modular group.
\begin{cor} \label{he5}
  \begin{align}
    q^{1/5}\sum_{n=0}^{\infty}   \Bigl ( \sum_{d \mid 5 n +1} d \Bigr )q^{n} &=  B^{9}A + 7 B^{4} A^{6}, \\ q^{2/5}\sum_{n=0}^{\infty}   \Bigl ( \sum_{d \mid 5 n +2} d \Bigr )q^{n} &=  3 B^{8} A^{2} - 4B^{3}A^{7}, \\  q^{3/5}\sum_{n=0}^{\infty}   \Bigl ( \sum_{d \mid 5 n +3} d \Bigr )q^{n} &= 4B^{7}A^{3} + 3 B^{2} A^{8},  \\ q^{4/5}\sum_{n=0}^{\infty}   \Bigl ( \sum_{d \mid 5 n +4} d \Bigr )q^{n} &= 7 B^{6} A^{4} -BA^{9}.
  \end{align}
\end{cor}

\begin{proof} The former identity of \eqref{vfe1} may be rewritten 
  \begin{align} \label{kfr}
    5E_{2}(q^{5}) - E_{2}(q) = 4 A^{10}(q) + 4 B^{10}(q).
  \end{align}
Apply Theorem \ref{gt} and the pentamidiation array $\mathcal{B}_{2}$ to \eqref{kfr} to conclude
\begin{align} \label{nyu}
  5E_{2}(q) - E_{2}(q^{1/5}) &= 4 \left(B^2+A^2\right) \left(B^8+6 B^7 A+17 B^6 A^2+18 B^5 A^3 \right. \\ & \qquad \qquad \qquad \quad \quad  \left. +25 B^4 A^4 -18 B^3 A^5+17 B^2 A^6-6 B A^7+A^8\right). \nonumber
\end{align}  
To derive each identity, proceed as in the proof of Corollary \ref{pr5}, and equate terms on each side of \eqref{nyu} of the form $q^{k/5}$ for each nonzero residue class $k$ modulo five.
\end{proof}
The parameterization on the right of \eqref{vfe1} and the left of \eqref{sp} lead to similar decompositions via the Fourier expansions for $E_{2,\left ( \frac{.}{5} \right )}(q)$, $L_{2,\chi_{3,5}}(q)$ and the identities
\begin{align}
E_{2,\left ( \frac{.}{5} \right )}(q^{1/5}) &= (B^{2} - AB - A^{2})^{5}, \label{oo1}  \\
  L_{2,\c\nonumber \chi_{3,5}}(q^{1/5}) &= AB^{9} + A^{2} B^{8} + 2 A^{3} B^{7} + 3 A^{4} B^{6} + 5 A^{5}B^{5} \\ & \qquad \qquad - 3 A^{6}B^{4} + 2A^{7} B^{3} - A^{8} B^{2} + A^{9} B. \label{oo2}
\end{align}
\begin{cor} \label{he6}
  \begin{align*}
    q^{1/5}\sum_{n=0}^{\infty}   \left ( \sum_{d \mid 5 n +1} \left ( \frac{d}{5} \right ) d \right )q^{n} &=   B^{9}A - 3 B^{4}A^{6} =   q^{1/5}\sum_{n=0}^{\infty}   \left ( \sum_{d \mid 5 n +1} \left ( \frac{(5n+1)/d}{5} \right ) d \right )q^{n} , 
\\ q^{2/5}\sum_{n=0}^{\infty}   \left ( \sum_{d \mid 5 n +2}  \left ( \frac{d}{5} \right ) d \right )q^{n} &=  - B^{8}A^{2} - 2 B^{3}A^{7} = -  q^{2/5}\sum_{n=0}^{\infty}   \left ( \sum_{d \mid 5 n +2}  \left ( \frac{(5n+2)/d}{5} \right ) d \right )q^{n}, \\  q^{3/5}\sum_{n=0}^{\infty}   \left ( \sum_{d \mid 5 n +3}  \left ( \frac{d}{5} \right ) d \right )q^{n} &= -2B^{7}A^{3} + B^{2}A^{8} = -q^{3/5}\sum_{n=0}^{\infty}   \left ( \sum_{d \mid 5 n +3}  \left ( \frac{(5n+3)/d}{5} \right ) d \right )q^{n},  \\ q^{4/5}\sum_{n=0}^{\infty}   \left ( \sum_{d \mid 5 n +4}  \left ( \frac{d}{5} \right ) d \right )q^{n} &= 3B^{6}A^{4}+BA^{9} = q^{4/5}\sum_{n=0}^{\infty}   \left ( \sum_{d \mid 5 n +4}  \left ( \frac{(5n+4)/d}{5} \right ) d \right )q^{n}.
  \end{align*}
\end{cor}

For the Eisenstein series on $SL(2, \Bbb Z)$ of weight four and six, both the Processes of Pentication and Pentamidiation yield homogeneous polynomials in $A(q)$ and $B(q)$. As a result, we obtain corresponding parameterizations for the modular discriminant
\begin{align} \label{jq}
  \Delta(q) = q(q;q)_{\infty}^{24} = \frac{E_{4}^{3}(q) - E_{6}^{2}(q)}{1728} = A^{5} B^{5}  (B^{10} - 11 A^{5} B^{5} - A^{10}  )^{5}.
\end{align}
\begin{cor} \label{wu} Let $B = B(q), A = A(q)$. Then 
 \begin{align} \label{by}
 E_{4}(q^{5}) &= B^{20} - 12 B^{15} A^{5} + 14 B^{10} A^{10} + 12 B^{5} A^{15} + A^{20}, \\     
E_{6}(q^{5}) &= B^{30} - 18 B^{20} A^{5} + 75 B^{20} A^{10} + 75 B^{10} A^{20} + 18 B^{5}A^{25} + A^{30}, \label{by1} \\
 & q^{5}(q^{5}; q^{5})_{\infty}^{24} = A^{25} B^{25} (B^{10} - 11 B^{5} A^{5} - A^{10}).
  \end{align}
  \end{cor}
  \begin{cor} \label{bn}
      \begin{align*}
  E_{4}(q^{1/5}) &= A^{20}-240 A^{19} B+2160 A^{18} B^2-6720 A^{17} B^3+17520 A^{16} B^4-30228 A^{15} B^5 \\ & +57840 A^{14} B^6-60960 A^{13} B^7+79920 A^{12} B^8-41520 A^{11} B^9 +60494 A^{10} B^{10} \\ & +41520 A^9 B^{11}+79920 A^8 B^{12}+60960 A^7 B^{13}+57840 A^6 B^{14} +30228 A^5 B^{15} \\ &+17520 A^4 B^{16}+6720 A^3 B^{17}+2160 A^2 B^{18}+240 A B^{19}+B^{20},      \end{align*}
\begin{align*}
  E_{6}(q^{1/5}) &= A^{30} + 504 A^{29} B-16632 A^{28} B^2+122976 A^{27} B^3-532728 A^{26} B^4  \\ & +1575522 A^{25} B^5 -4049640 A^{24} B^6+8205120 A^{23} B^7 -15203160 A^{22} B^8  \\ & +22425480 A^{21} B^9 -31510005 A^{20} B^{10}+32502960 A^{19} B^{11}-37633680 A^{18} B^{12}  \\ & +21450240 A^{17} B^{13} -26046720 A^{16} B^{14}-26046720 A^{14} B^{16}-21450240 A^{13} B^{17} \\ & -37633680 A^{12} B^{18}  -32502960 A^{11} B^{19}-31510005 A^{10} B^{20}-22425480 A^9 B^{21} \\ & -15203160 A^8 B^{22} -8205120 A^7 B^{23}-4049640 A^6 B^{24}-1575522 A^5 B^{25} \\ & -532728 A^4 B^{26} -122976 A^3 B^{27}-16632 A^2 B^{28}-504 A B^{29}+B^{30},
\end{align*}
  \begin{align}
q^{1/5} (q^{1/5}; q^{1/5})_{\infty}^{24} &= AB(B^{10} - 11 A^{5}B^{5} - A^{10})(B^{2} - AB - A^{2})^{24}. \label{nv}
  \end{align}
\end{cor}
The formulas of Corollary \ref{wu} lead to two identities relating the coefficients from \eqref{g1}--\eqref{g2} and \eqref{by}--\eqref{by1} with the inverses of $\mathcal{A}_{4}$ and $\mathcal{A}_{6}$ defined in Corollary \ref{tipp}
\begin{align}
  \mathcal{A}_{4}^{-1}
  \begin{pmatrix}
1 & 228 & 494 & -228 & 1      
  \end{pmatrix}^{T} 
&= 
\begin{pmatrix}
  1 & -12 & 14 & 12 & 1
\end{pmatrix}^{T}, \\
\mathcal{A}_{6}^{-1}  \begin{pmatrix}
   1 & -522 & -10005 & -10005  & 522 & 1
\end{pmatrix}^{T}
&= 
\begin{pmatrix}
   1 & -18 &  75  & 75 &  18 & 1
\end{pmatrix}^{T}. 
\end{align}
The next corollary, deducible by induction from Theorem \ref{gt}, provides a more general interpretation for the action of the inverse Hecke operators on the Eisenstein series of level one.
\begin{cor}
  \begin{align}
    \Omega_{5,0}^{-n} \Bigl ( E_{k}(q) \Bigr ) = \frac{(1 - 5^{(k-1)n})E_{k}(q^{5}) - (1 - 5^{(k-1)(n-1)})E_{k}(q)}{5^{(k-1)(n-1)} - 5^{(k-1)n}}.
  \end{align}
\end{cor}

In the following two corollaries, we apply parameterizations from Corollary \ref{bn} to obtain corresponding quintic decompositions for the series $E_{4}(q)$ and $\Delta(q)$.
\begin{cor} Let $A = A(q)$ and $B = B(q)$. Then
  \begin{align*}
   q^{4/5} &\sum_{n=0}^{\infty} \Bigl ( \sum_{d \mid 5 n +4} d^{3} \Bigr ) q^{n} = 73 A^{4} B^{16} + 173A^{9} B^{11} + 241 A^{14} B^{6} - A^{19}B, 
\\  q^{3/5} &\sum_{n=0}^{\infty} \Bigl ( \sum_{d \mid 5 n +3} d^{3} \Bigr ) q^{n} = 
9A^{18} B^{2} -254A^{13} B^{7}+333 A^{8} B^{12} +28 A^{3} B^{17}, \\ 
 q^{2/5} &\sum_{n=0}^{\infty} \Bigl ( \sum_{d \mid 5 n +2} d^{3} \Bigr ) q^{n} = 
9 A^{2} B^{18}+254 A^{7}B^{13} +333 A^{12} B^{8} -28A^{17} B^{3}, \\ 
 q^{1/5} &\sum_{n=0}^{\infty} \Bigl ( \sum_{d \mid 5 n +1} d^{3} \Bigr ) q^{n} = 73A^{16} B^{4} -173A^{11} B{^9} +241A^{6} B^{14} + A B^{19}.
  \end{align*}
\end{cor}

\begin{cor} Let $A = A(q)$ and $B = B(q)$ and let $\tau(n)$ denote the coefficients of the modular discriminant function $\Delta(q) = q(q;q)_{\infty}^{24} = \sum_{n=1}^{\infty} \tau(n)q^{n}$.  Then 
  \begin{align*}
   q^{4/5} \sum\nolimits_{n=0}^{\infty} \tau(5n+4) q^{n} &= -1472 A^4 B^{56}-66539 A^9 B^{51}+2946185 A^{14} B^{46} \\ & -22164065 A^{19} B^{41} -41743460 A^{24} B^{36}+426443402 A^{29} B^{31} \\ & -21441266 A^{34} B^{26} -83214230 A^{39} B^{21} -14213080 A^{44} B^{16} \\ & -716495 A^{49} B^{11} -6083 A^{54} B^6 -A^{59} B, 
\\  q^{3/5} \sum\nolimits_{n=0}^{\infty} \tau(5n+3) q^{n} &= 252 A^3 B^{57}+76164 A^8 B^{52}-2846410 A^{13} B^{47}B^{41} \\ & +37021490 A^{18} B^{42} -179092490 A^{23} B^{37}+86337768 A^{28} B^{32}  \\ & +377079066 A^{33} B^{27} +71302530 A^{38} B^{22}+2608430 A^{43} B^{17} \\ &+104720 A^{48} B^{12}+15928 A^{53} B^7 -24 A^{58} B^2 \\ 
 q^{2/5} \sum\nolimits_{n=0}^{\infty} \tau(5n+2) q^{n} &= -24 A^2 B^{58}-15928 A^7 B^{53}+104720 A^{12} B^{48}-2608430 A^{17} B^{43} \\ & +71302530 A^{22} B^{38}-377079066 A^{27} B^{33}+86337768 A^{32} B^{28} \\ &+179092490 A^{37} B^{23}+37021490 A^{42} B^{18}+2846410 A^{47} B^{13} \\ & +76164 A^{52} B^8-252 A^{57} B^3, \\ 
 q^{1/5} \sum\nolimits_{n=0}^{\infty} \tau(5n+1) q^{n} &= A B^{59}-6083 A^6 B^{54}+716495 A^{11} B^{49}-14213080 A^{16} B^{44} \\ & +83214230 A^{21} B^{39}-21441266 A^{26} B^{34}-426443402 A^{31} B^{29}\\ &-41743460 A^{36} B^{24}+22164065 A^{41} B^{19}+2946185 A^{46} B^{14}\\ &+66539 A^{51} B^9-1472 A^{56} B^4.
\end{align*}
\end{cor}
By proceeding as above, we may show that Corollary \ref{tipp} induces similar multisections for modular forms of any weight on the full modular group or on subgroups of level five. 
The quintic decompositions given in the preceding corollaries yield a wealth of combinatorial and analytic consequences. The next result follows from Corollary \ref{pr5}.
\begin{cor} \label{vrw}
For each integer $n \ge 0$, 
\begin{align}
   \sum_{d \mid 5n+4} \re \chi_{4,5}(d)  =  \sum_{d \mid 5n+1} \im \chi_{4,5}(d) = 0
\end{align}  
\end{cor}
The claims of Corollary \ref{vrw} may be justified by restricting $-1< q< 1$ in Corollary \ref{pr5} and equating real and imaginary parts of the resulting identities. Corollary \ref{pr5} also yields several curious representations for the Rogers-Ramanujan continued fraction.
\begin{cor} Let $R(q)$ denote the Rogers-Ramanujan continued fraction defined by \eqref{d1}. Then
  \begin{align}
   q^{-1/5}R(q) 
&= \frac{1+i}{1 -i}\sum_{n=0}^{\infty} \left ( \sum_{d \mid 5n+3} \chi_{4,5}(d) \right ) q^{n} /  \sum_{n=0}^{\infty} \left ( \sum_{d \mid 5n+2}  \chi_{4,5}(d) \right ) q^{n},
  \end{align}
  \begin{align}
   q^{-3/5}R^{3}(q) &=  -i\sum_{n=0}^{\infty} \left ( \sum_{d \mid 5n+4} \chi_{4,5}(d) \right ) q^{n} /  \sum_{n=0}^{\infty} \left ( \sum_{d \mid 5n+1} \chi_{4,5}(d) \right ) q^{n}, 
\end{align}
\begin{align}
   \frac{(q^{5}; q^{5})_{\infty}^{5}}{(q;q)_{\infty}} &= \frac{1}{2} \left (  \sum_{n=0}^{\infty} \left ( \sum_{d \mid 5n+3} \chi_{4,5}(d) \right ) q^{n}\right ) \left (  \sum_{n=0}^{\infty} \left ( \sum_{d \mid 5n+2}  \chi_{4,5}(d) \right ) q^{n} \right ) \\ &=  - i\left (  \sum_{n=0}^{\infty} \left ( \sum_{d \mid 5n+1} \chi_{4,5}(d) \right ) q^{n}\right ) \left (  \sum_{n=0}^{\infty} \left ( \sum_{d \mid 5n+4}  \chi_{4,5}(d) \right ) q^{n} \right ).
\end{align}
\end{cor}
We may reduce the formulas from Corollaries \ref{he5} and \ref{he6} modulo $2, 3$, and $7$.
\begin{cor}
  \begin{align}
    \sum_{n=0}^{\infty} \Bigl ( \sum_{d \mid 5n+1} d \Bigr )q^{n} &\equiv \frac{(q;q)_{\infty}^{3} (q^{5}; q^{5})_{\infty}}{(q;q^{5})_{\infty}^{8}(q^{4}; q^{5})_{\infty}^{8}} \pmod{7}, \\   \sum_{n=0}^{\infty} \Bigl ( \sum_{d \mid 5n+3} d \Bigr )q^{n} &\equiv \frac{(q;q)_{\infty} (q^{5}; q^{5})_{\infty}^{3}}{(q;q^{5})_{\infty}^{4}(q^{4}; q^{5})_{\infty}^{4}} \pmod{3},  \\   \sum_{n=0}^{\infty} \Bigl ( \sum_{d \mid 5n+2} d \Bigr )q^{n} 
&\equiv  
\frac{(q;q)_{\infty}^{2} (q^{5}; q^{5})_{\infty}^{2}}{(q;q^{5})_{\infty}^{6}(q^{4}; q^{5})_{\infty}^{6}} \pmod{2}, \\ 
\sum_{n=0}^{\infty}   \Bigl ( \sum_{d \mid 5 n +1} \left ( \frac{d}{5} \right ) d \Bigr )q^{n} &\equiv  \frac{(q;q)_{\infty}^{3} (q^{5}; q^{5})_{\infty}}{(q;q^{5})_{\infty}^{8}(q^{4}; q^{5})_{\infty}^{8}} \pmod{3},
 \\ \sum_{n=0}^{\infty}   \Bigl ( \sum_{d \mid 5 n +3} \left ( \frac{d}{5} \right ) d \Bigr )q^{n} &\equiv  \frac{(q;q)_{\infty}^{2} (q^{5}; q^{5})_{\infty}^{2}}{(q^{2};q^{5})_{\infty}^{6}(q^{3}; q^{5})_{\infty}^{6}} \pmod{2},
\\ \sum_{n=0}^{\infty}   \Bigl ( \sum_{d \mid 5 n +4} \left ( \frac{d}{5} \right ) d \Bigr )q^{n} &\equiv   \frac{(q;q)_{\infty}^{3} (q^{5}; q^{5})_{\infty}}{(q^{2};q^{5})_{\infty}^{8}(q^{3}; q^{5})_{\infty}^{8}} \pmod{3}.   \end{align}
\end{cor}

We now turn our attention to parameterizations for quintic multisections corresponding to the trivial residue class modulo five. Formulas for successive applications of $\Omega_{5, 0}$ may be derived from the initial parameterizations for each series and Corollary \ref{tipp}. Since the ring of modular forms for the full modular group is generated by $E_{4}(q)$ and $E_{6}(q)$, Theorem \ref{hut} represents a special case of a more general theorem for quintic multisections of modular forms. Theorem \ref{wheel} is a consequence of Corollary \ref{tipp}.
\begin{thm} \label{wheel}
  Suppose $f(\tau)$, $q = e^{2 \pi i \tau}$ is a modular form of weight $d$; the $n$th iterate of the multisection operator $\Omega_{5,0}$ is denoted $\Omega_{5,0}^{(n)} = \Omega_{5,0} \circ \Omega_{5,0} \circ \cdots$; and $\mathcal{A}_{d}$ is defined as in Corollary \ref{tipp}. Then there exist integers $a_{j}$, $0 \le j \le d$ such that 
  \begin{align}
  f(\tau) =  \sum_{k=0}^{d} a_{k}A^{5k}(q) B^{5(d - k)}(q).
  \end{align}
Moreover, if $\begin{pmatrix}
  b_{0,n} & b_{1,n} & \cdots &  b_{d,n}
\end{pmatrix}^{T}  = \mathcal{A}_{d}^{n}
\begin{pmatrix}
  a_{0} & a_{1} & \cdots & a_{d}
\end{pmatrix}^{T}$, then
\begin{align*}
  \Omega_{5,0}^{(n)}f(\tau) =  \sum_{k=0}^{d} b_{k,n}A^{5k}(q) B^{5(d - k)}(q).
\end{align*}
\end{thm}
The Eisenstein series on $\Gamma_{0}(5)$ are easily seen to be invariant under the operator $\Omega_{5,0}$. 
\begin{thm}
  Let $E_{k, \chi}(q)$ be the Eisenstein series on $\Gamma_{0}(5)$ defined by \eqref{eisdef}. Then 
  \begin{align}
   \Omega_{5,0}^{(n)}E_{k, \chi}(q) = E_{k, \chi}(q). 
  \end{align}
\end{thm}
\begin{proof}
 The invariance may be deduced from the fact $\chi(m) = 0$ for $m \equiv 0 \pmod{5}$. Therefore,
\begin{align}
\sum_{d \mid 5^{m} n}  \chi(d)  d^{k} = \sum_{d \mid n} \chi(d)  d^{k}. 
\end{align}
In other words, we have shown that $\lambda = 1$ is an eigenvalue for $\mathcal{A}_{d}$ for each $d \in \Bbb N$.
\end{proof}
The operator $\Omega_{5, 0}$ behaves very differently on the Eisenstein series for $\Gamma_{0}(5)$ when we consider the Fourier expansions about zero. In particular, depending on the weight, their quintic parameterizations correspond to eigenvalues for $\mathcal{A}_{k}$ of the form $5^{k}$, $k \in \Bbb N$. 
\begin{thm} \label{hll}
  For all $k \in \Bbb N$ and Dirichlet characters $\chi$ modulo five,  
  \begin{align}
\Omega_{5,0}L_{k, \chi}(q) = 5^{k+1} L_{k, \chi}(q).    
  \end{align}
\end{thm}
\begin{proof}
  We show the result by induction on $k$, making use of the differential equation 
\begin{align}
 (\wp'(\theta))^{2} &= 4 \wp^{3}(\theta ) - \frac{1}{12}E_{4}(q) \wp(\theta) - \frac{1}{216}E_{6}(q), \qquad \wp(\theta) = -\zeta'(\theta)  \label{ell2}
\end{align}
for the Weierstrass $\wp$-function. We will also need the representations \cite[Lemma 3.8]{qeis}
\begin{align} \label{h4}
  \wp(\theta + 6 \pi \tau \mid 5 \tau) &= - \frac{E_{2}(q^{5})}{12} - \sum_{n=1}^{\infty} \frac{n (q^{3n} e^{i n \theta} + q^{2n} e^{- 2 i n \theta})}{1 - q^{5n}}, \\ 
  \wp(\theta + 8 \pi \tau \mid 5 \tau) &= - \frac{E_{2}(q^{5})}{12} - \sum_{n=1}^{\infty} \frac{n (q^{4n} e^{i n \theta} + q^{n} e^{- 2 i n \theta})}{1 - q^{5n}}. \label{h5}
\end{align}
In particular, Theorem \ref{tipp} may be applied to parameterizations from \cite[Lemma 3.9]{qeis}
\begin{align}
  \Omega_{5,0} & \left ( 1 - \frac{1}{2} \sum_{n=1}^{\infty} \frac{q^{n} - q^{4n}}{1 - q^{5n}}\right ) = \Omega_{5,0} \left ( \frac{23 - 3B^{5} - A^{5}}{12} \right ) = \frac{23 - 3B^{5} - A^{5}}{12},  \\ 
 \Omega_{5,0} & \left ( 1 - \frac{1}{2} \sum_{n=1}^{\infty} \frac{q^{2n} - q^{3n}}{1 - q^{5n}}\right ) = \Omega_{5,0} \left ( \frac{23 + 3B^{5} - A^{5}}{12} \right ) = \frac{23 + 3B^{5} - A^{5}}{12}, 
\end{align}
 \begin{align} \label{bc1}
\Omega_{5,0}  \Bigl ( \wp( 6 \pi \tau \mid 5 \tau) \Bigr ) &= \Omega_{5,0} \left ( - \frac{B^{10} - 6 A^{5}B^{5} + A^{10}}{12} \right ) \\ &= 5 \left ( - \frac{B^{10} - 6 A^{5}B^{5} + A^{10}}{12} \right ) = 5 \wp( 6 \pi \tau \mid 5 \tau), \\    
\Omega_{5,0}  \Bigl ( \wp( 8 \pi \tau \mid 5 \tau) \Bigr ) &= \Omega_{5,0} \left ( - \frac{B^{10} + 6 A^{5}B^{5} + A^{10}}{12} \right ) \\ &= 5 \left ( - \frac{B^{10} + 6 A^{5}B^{5} + A^{10}}{12} \right ) = 5 \wp( 8 \pi \tau \mid 5 \tau),
\\    
\Omega_{5,0}  \Bigl ( \wp'( 8 \pi \tau \mid 5 \tau) \Bigr ) &= \Omega_{5,0} \left ( iA^{5}B^{10} \right ) = 25iA^{5}B^{10}  = 25 \wp( 8 \pi \tau \mid 5 \tau), \\ 
\Omega_{5,0}  \Bigl ( \wp'( 6 \pi \tau \mid 5 \tau) \Bigr ) &= \Omega_{5,0} \left ( iB^{5}A^{10} \right ) = 25iB^{5}A^{10}  = 25 \wp( 6 \pi \tau \mid 5 \tau), 
\end{align}
Similarly, from \eqref{ell2},\eqref{by}--\eqref{by1}, and Theorem \ref{tipp} we derive
\begin{align}
\Omega_{5,0}  \Bigl ( \wp''( 8 \pi \tau \mid 5 \tau) \Bigr ) &= \Omega_{5,0} \left ( B^{10}A^{10} - A^{5}B^{15} \right ) = 125 \wp''( 8 \pi \tau \mid 5 \tau), \\ 
\Omega_{5,0}  \Bigl ( \wp''( 6 \pi \tau \mid 5 \tau) \Bigr ) &= \Omega_{5,0} \left ( A^{10}B^{10} - B^{5}A^{15} \right ) = 125 \wp''( 6 \pi \tau \mid 5 \tau). \label{bc5}
 \end{align}
We may iteratively differentiate \eqref{ell2} with respect to $\theta$ to obtain
\begin{align} \label{conv}
  \wp^{(n+2)}(\theta) = 6 \sum_{k=0}^{n} {n \choose k} \wp^{(k)}(\theta) \wp^{(n - k)}(\theta), \qquad n \ge 1.
\end{align}
By proceeding inductively through \eqref{conv}, \eqref{h4}--\eqref{h5}, and \eqref{bc1}--\eqref{bc5}, we deduce
\begin{align} \label{bv1}
  \Omega_{5,0} &\left ( \sum_{n=1}^{\infty} n^{k}\frac{q^{3n} + (-1)^{k+1} q^{2n}}{1 - q^{5n}} \right ) = 5^{k} \sum_{n=1}^{\infty} n^{k}\frac{q^{3n} + (-1)^{k+1} q^{2n}}{1 - q^{5n}}, \\ \Omega_{5,0} &\left ( \sum_{n=1}^{\infty} n^{k}\frac{q^{4n} + (-1)^{k+1} q^{n}}{1 - q^{5n}} \right ) = 5^{k} \sum_{n=1}^{\infty} n^{k}\frac{q^{4n} + (-1)^{k+1} q^{n}}{1 - q^{5n}}. \label{bv2}
\end{align}
For an appropriate Dirichlet character modulo five, each series $L_{k+1, \chi}(q)$ may be represented as a linear combination of \eqref{bv1}--\eqref{bv2}. Since the operator $\Omega_{5,0}$ acts linearly on these series, we have established the claimed result.
\end{proof}

Ramanujan \cite{MR1701582} proved the special case $k = 1$ in Theorem \ref{hll} and used it to deduce (1.13b) and the first order congruence  relation for the partition function modulo five. 
\begin{cor} For each $m,k \in \Bbb N$, 
  \begin{align}\label{hy}
   \sum_{d \mid 5^{m}n}  \left ( \frac{(5^{m}n)/d}{5} \right ) d^{k}  = 5^{mk} \sum_{d \mid n}  \left ( \frac{n/d}{5} \right ) d^{k}.
\end{align}
\end{cor}

The weight-two case of Theorem \ref{hll} contains further combinatorial content. For a given partition $\lambda$, each square in the Young diagram representation for $\lambda$ defines a \textit{hook} consisting of that square, all the squares to the right of that square, and all the squares below that square. The \textit{hook number} of a given square is the number of squares in that hook. A partition $\lambda$ is said to be a t-core if it has no hook numbers that are multiples of t.  From Theorem \ref{hll} and the fact (c.f. \cite{MR1055707}) that the generating function for the number of $5$-cores is $$\sum_{n=0}^{\infty} c_{t}(n) q^{n} = \frac{(q^{5}; q^{5})_{\infty}^{5}}{(q;q)_{\infty}},$$ we deduce a product formula for the generating function for $c_{5}(n)$ and a corresponding congruence. The latter congruence of \eqref{nhj} is also deduced in \cite[Corollary 1, p. 11]{MR1055707}.
\begin{cor}
\begin{align} \label{nhj}
\sum_{m=0}^{\infty} c_{5}(5^{n}m -1) q^{m} =   5^{n}q\frac{(q^{5}; q^{5})_{\infty}^{5}}{(q;q)_{\infty}}, \qquad c_{5}(5^{n}m -1) \equiv 0 \pmod{5^{n}}.
\end{align}  
\end{cor}

We next show similarly that the quintic multisections $\Omega_{5,0}$ for the modular discriminant $\Delta(q) =  q (q;q)_{\infty}^{24}$ may be represented in compact form as a polynomial of fixed degree whose coefficients arise from a generating matrix of smaller dimension than expected from Theorem \ref{tipp}.
Parameterizations for quintic multisections have the same basic shape, apart from a polynomial factor of degree $40$ in the quintic theta functions.
\begin{thm} \label{bgr}
  Let $F(q) = q (q;q)_{\infty}^{24}$, $x = B(q)$, $y = A(q)$. Then for $n \in \Bbb N \cup \{0\}$, 
  \begin{align*}
\sum_{m=0}^{\infty} \tau(5^{n}m)q^{m} 
= \Omega_{5,0}(F(q)) = 5^{n}x^{5} y^{5}(x^{10} - 11 x^{5} y^{5} - y^{10})P_{n}(x^{5}, y^{5}),
  \end{align*}
where $P_{n}(u, v)$ is a homogeneous polynomial of degree $40$ with integer coefficients 
\begin{align*}
  P_{n}(x^{5}, y^{5}) = a_{1,n}x^{40} - a_{2,n}x^{35} y^5 &+a_{3,n} x^{30} y^{10}- a_{4,n} x^{25} y^{15}+a_{5,n} x^{20} y^{20} \\ &+a_{4,n}x^{15} y^{25}+a_{3,n}x^{10} y^{30}+a_{2,n} x^5 y^{35}+a_{1,n} y^{40}.
\end{align*}
The coefficients of the respective polynomials $P_{n}(x^{5}, y^{5})$ are defined by
\begin{align}
  a_{1,0} = 1, \quad a_{2,0} = 44, \quad a_{3,0} = 722, \quad a_{4,0} = 5192, \quad a_{5,0} = 13195
\end{align}
and, for $n \ge 1$, by
\begin{align}
  \begin{pmatrix}
    a_{1,n} \\ a_{2,n} \\ a_{3,n} \\ a_{4,n} \\ a_{5,n} 
  \end{pmatrix}
 =
 \begin{pmatrix}
   893  & -88 & 9 & -1 & 1/5 \\ 
   66352 & -5702 & 506 & -49 & 44/5 \\
   1081476 & -103301 & 9503 & -887 & 722/5 \\ 
  7778056 & -740071 & 72578 & -6992 & 5192/5 \\  
   19767160 & -1881110 & 183530 & -19520 & 2639 
 \end{pmatrix}^{n}
 \begin{pmatrix}
   1 \\ 44 \\ 722 \\ 5192 \\ 13195
 \end{pmatrix}.
\end{align}
\end{thm}

%
\begin{proof}
  We prove the result by induction on $n$. The cases $n = 0,1$ follow from \eqref{jq} and \eqref{nv}, respectively. The inductive step follows from Theorem \ref{quint}, which implies
  \begin{align}
   \Omega_{5,0} \Bigl (L_{2}(q)E_{2, \chi_{3,5}}(q)P_{n}(A^{5}, B^{5}) \Bigr ) = 
L_{2}(q)E_{2, \chi_{3,5}}(q)\sum_{k=1}^{9} a_{k,n+1}B^{5k}A^{40-5k}, \nonumber
  \end{align}
where
\begin{align*}
  a_{1,n+1} &= 4465 a_{1,n}-440 a_{2,n}+45 a_{3,n}-5 a_{4,n}+a_{5,n}, \\  a_{2,n+1} &= -331760 a_{1,n}+28510 a_{2,n}-2530 a_{3,n}+245 a_{4,n}-44 a_{5,n},  \\ a_{3,n+1} &= 5407380 a_{1,n}-516505 a_{2,n}+47515 a_{3,n}-4435 a_{4,n}+722 a_{5,n}, \\  a_{4,n+1} &= -38890280 a_{1,n}+3700355 a_{2,n}-362890 a_{3,n}+34960 a_{4,n}-5192 a_{5,n}, \\ a_{5,n+1} &= 98835800 a_{1,n}-9405550 a_{2,n}+917650 a_{3,n}-97600 a_{4,n}+13195 a_{5,n}, \\ a_{6,n+1} &= - a_{4,n+1}, \quad a_{7,n+1} = a_{3,n+1}, \quad a_{b,n+1} = - a_{2,n+1}, \quad a_{9,n+1} = a_{1,n+1}.
\end{align*}
The claimed matrix formulation for the coefficients of $P_{n}(A^{5}, B^{5})$ may be deduced by extracting a factor of $5$ from the right side of each of the preceding relations.
\end{proof}

These relations imply a set of congruences modulo $5^{n}$ for Ramanujan's $\tau$-function. 
\begin{cor} Let $\sum_{n=1}^{\infty} \tau(n) q^{n} = q (q;q)_{\infty}^{24}.$ Then 
  \begin{align}
    \tau(5^{n}m) \equiv 0 \pmod{5^{n}}, \qquad m, n \in \Bbb N.
  \end{align}
\end{cor}

Each preceding relation between series of argument $q$ and series of argument $q^{5}$ may be interpreted via \eqref{rog} as a relation between the Rogers-Ramanujan functions. 
A quartet of relations between the Rogers-Ramanujan functions results when the generating identities \eqref{d00}-\eqref{po} are rewritten in terms of the Rogers-Ramanujan functions.
\begin{cor} \label{indus} If $\alpha, \beta$ are given by Theorem \ref{mainthm}, and $G(q)$ and $H(q)$ by \eqref{rog}, then 
\begin{align}
\sqrt{5}qH(q^{5}) &= \sqrt[5]{\frac{G^{3}(q)}{H^{2}(q)} - \beta^{5} q\frac{H^{3}(q)}{G^{2}(q)}} - \sqrt[5]{\frac{G^{3}(q)}{H^{2}(q)} - \alpha^{5} q\frac{H^{3}(q)}{G^{2}(q)}}, \\ 
\sqrt{5}G(q^{5}) &= \alpha \sqrt[5]{\frac{G^{3}(q)}{H^{2}(q)} - \beta^{5} q\frac{H^{3}(q)}{G^{2}(q)}} - \beta \sqrt[5]{\frac{G^{3}(q)}{H^{2}(q)} - \alpha^{5} q\frac{H^{3}(q)}{G^{2}(q)}}, \\
 \frac{H^{3}(q)}{G^{2}(q)} &= q^{4}H^{5}(q^{5}) - 3 q^{3}H^{4}(q^{5})G(q^{5}) + 4 q^{2} H^{3}(q^{5}) G^{2}(q^{5}) \nonumber \\ & \qquad - 2 q H^{2}(q^{5}) G^{3}(q^{5}) +  H(q^{5}) G^{4}(q^{5}), \label{fa1} \\ 
\frac{G^{3}(q)}{H^{2}(q)} &= G^{5}(q^{5}) + 3 q G^{4}(q^{5})H(q^{5}) + 4 q^{2} G^{3}(q^{5}) H^{2}(q^{5}) \nonumber \\ & \qquad - 2 q^{3} G^{2}(q^{5}) H^{3}(q^{5}) +  q^{4}G(q^{5}) H^{4}(q^{5}). \label{fa2}
\end{align}
\end{cor}
Equivalent formulations of identities \eqref{fa1}-\eqref{fa2} appear in \cite{gugg}. These notably induce a famous modular equation of Ramanujan \cite[p. xxvii]{ram_coll} which, in turn, M. Hirschorn and D. Hunt  \cite{MR622342} employed to prove Ramanujan's congruences \eqref{par}. 
We derive formulas for quintic multisections of the partition function in the next section.

The next identity will be of particular importance in the following two sections.
\begin{lem} \label{mos} 
  \begin{align} \label{bs}
     \frac{E_{2, \chi_{3, 5}}(q^{1/5})}{L_{2, \chi_{3,5}}(q)} = \left (\frac{E_{2, \chi_{3, 5}}(q)}{L_{2,\chi_{3,5}}(q^{1/5})} \right )^{5}.
  \end{align}
\end{lem}
\begin{proof}
Apply Theorems \ref{quint} and \ref{par1}, \eqref{oo1}, and the elementary identity
\begin{align} \label{fy1}
  A^{10}+ 11A ^5B^5- B^{10} &= \left(A^2+A B-B^2\right) \left(A^4-3 A^3 B+4 A^2 B^2-2 A B^3+B^4\right) \nonumber \\ & \qquad \qquad \times \left(A^4+2 A^3 B+4 A^2 B^2+3 A B^3+B^4\right).
\end{align}
Now multiply both sides by $AB$. Take the fifth power of the resulting equation, and apply relevant formulas from Section \ref{s4} to derive \eqref{bs}.
\end{proof}

\section{Quintic decomposition of the ordinary Partition function} \label{partition}
The dissection techniques from the previous sections may be applied to obtain expansions for corresponding dissections of the generating function for ordinary partitions. These dissections yield generating function witnesses for Ramanujan's congruences \eqref{par}. 
Watson is credited with the first formal proof of \eqref{par}. His methods were later simplified by A.O.L Atkin \cite{MR0233777}.  
The methods of Watson and Atkin \cite{MR0233777} for deriving (1.13b)--\eqref{bvf1}, and corresponding higher-order generating functions, depend fundamentally on the remarkable modular equation \cite[p. 119]{knopp}
\begin{align} \label{it}
 q\frac{(q^{5};q^{5})_{\infty}^{6}}{(q; q)_{\infty}^{6}} = & 5^{2}q^{5}\frac{(q^{25}; q^{25})_{\infty}^{5}}{(q;q)_{\infty}^{5}} +  5^{2}q^{4}\frac{(q^{25}; q^{25})_{\infty}^{4}}{(q;q)_{\infty}^{4}} \\ &+ 15q^{3}\frac{(q^{25}; q^{25})_{\infty}^{3}}{(q;q)_{\infty}^{3}} +  5q^{2}\frac{(q^{25}; q^{25})_{\infty}^{2}}{(q;q)_{\infty}^{2}} +  q\frac{(q^{25}; q^{25})_{\infty}}{(q;q)_{\infty}}. \nonumber
\end{align}
Although Ramanujan could have derived \eqref{it} from identities equivalent to Theorem \ref{quint}, as we will accomplish in the next section, Ramanujan likely did not use \eqref{it} to obtain quintic decompositions of the partition function.  Indeed, equation \eqref{it} is conspicuously absent from his manuscript \cite{MR1701582}. It is more likely that Ramanujan formulated expansions (1.13b)--\eqref{bvf1} from relations between Eisenstein series through an approach similar to that subsequently demonstrated. Ramanujan's narrative indicates that he was aware of a systematic technique for expanding the image of Eisenstein series under the operator $\pi$ defined by \eqref{vfn}. In particular, Ramanujan's calculations allude to a general inductive technique through which expansions may be derived for the multisections (1.13b)--\eqref{bvf1}.
He indicated that, by manipulating each successive identity and replacing $q$ by $q^{1/5}$, one may obtain expansions of higher order.
In this section, we presume to fill in the details that Ramanujan omitted. Our goal is to explicate Ramanujan's method for constructing relevant expansions from iterated dissections of Eisenstein series. Although our calculations refer principally to Theorems \ref{gt} and \ref{quint}, our technique depends ultimately on identities stated by Ramanujan, namely \eqref{gj6}--\eqref{po2}, expressed in the form of Theorem \ref{mainthm}. In addition to the classical formulas for the partition function, our methods result in quintic expansions for the series
\begin{align} \label{bgt}
  \sum_{n=0}^{\infty} p(5^{\lambda}n + \delta_{\lambda} + 5^{\lambda-1}m)q^{n}, \quad m = 0, 1, 2, 3, 4.
\end{align}
\begin{cor} \label{dve} 
  \begin{align}
   q\sum_{n=0}^{\infty} p(5 n+4)q^{n} &= \frac{5A^{5}B^{5}}{(q;q)_{\infty}^{5}} = 5 \frac{q(q^{5}; q^{5})_{\infty}^{5}}{(q;q)_{\infty}^{6}}, \label{frq}
\end{align}
\begin{align}   
     q^{1/5}\sum_{n=0}^{\infty} p(5 n) q^{n} &= \frac{AB^{9} - 3 A^{6} B^{4}}{(q;q)_{\infty}^{5}}, \quad  
     q^{2/5}\sum_{n=0}^{\infty} p(5 n +1) q^{n} = \frac{A^{2}B^{8} +2 A^{7} B^{3}}{(q;q)_{\infty}^{5}}, \\  
     q^{3/5}\sum_{n=0}^{\infty} p(5 n +2) q^{n} &= \frac{2A^{3}B^{7} - A^{8} B^{2}}{(q;q)_{\infty}^{5}}, \quad
     q^{4/5}\sum_{n=0}^{\infty} p(5 n +3) q^{n} = \frac{3A^{4}B^{6} +  A^{9} B}{(q;q)_{\infty}^{5}}. 
  \end{align}
\end{cor}

\begin{proof}
First note that, by \eqref{sp},
\begin{align}
  L_{2,\chi_{3,5}}(q^{1/5})
&= \frac{q^{1/5} (q; q)_{\infty}^{5}}{(q^{1/5};q^{1/5})_{\infty}} 
=  (q;q)_{\infty}^{5} \sum_{n=1}^{\infty} p(n-1) q^{n/5}. \label{mnb}
\end{align}
Equating terms with integral exponents in expansions \eqref{mnb} and \eqref{oo2} yields
\begin{align} \label{on1}
(q;q)_{\infty}^{5} \sum_{n=0}^{\infty} p(5n +4) q^{n} = 5 A^{5}B^{5} = \Omega_{5,0} \Bigl (  L_{2, \chi_{3,5}}(q) \Bigr ).
\end{align}
The third equality of \eqref{frq} follows from the product representations \eqref{dc} for $A(q)$ and $B(q)$.
The remaining equalities of Corollary \ref{dve} may be derived similarly by equating terms in \eqref{mnb} and \eqref{oo2} of the form $q^{(5n+k)/5}$ for $k$ in each residue class modulo five.
\end{proof}

The technique used to generate Corollary \ref{dve} may be iterated to produce corresponding parameterizations for $(q;q)_{\infty}\sum_{n=0}^{\infty}p(25n + 24 + m)q^{n}$, $m = 0, 1, 2, 3, 4$ via quintic multisection of Eisenstein series. If $m = 0$, we recover the classical expansion \eqref{bvf1}.  
\begin{thm} \label{vqw} Let $\Omega_{5,0}$ be defined as in \eqref{cdq1}. For $m = 0, 1, 2, 3, 4$, 
  \begin{align} \label{aco}
  q^{m/5}(q; q)_{\infty}&\sum_{n=0}^{\infty} p(25n +24+5m)q^{n} = \frac{\Omega_{5,m} \Bigl ( L_{2, \chi_{3}}^{5}(q) \cdot \Omega_{5,0}( L_{2, \chi_{3}}(q)) \Bigr )}{E_{2, \chi_{3,5}}^{5}(q)L_{2, \chi_{3, 5}}(q)}, 
  \end{align}
Moreover, when $m = 0$, identity \eqref{aco} is equivalent to Ramanujan's expansion \eqref{bvf1}.
\end{thm}

Divide both sides of \eqref{on1} by $E_{2, \chi_{3,5}}(q)$ and set $q$ to $q^{1/5}$ in the resultant to derive
\begin{align} \label{lev}
  (q; q)_{\infty}\sum_{n=1}^{\infty} p(5n -1)q^{n/5} &= \pi \left ( \frac{ \Omega_{5,0}\Bigl (L_{2,\chi_{3,5}}(q) \Bigr )}{E_{2, \chi_{3,5}}(q)} \right )  = \frac{ \pi \circ \Omega_{5,0}\Bigl (L_{2,\chi_{3,5}}(q) \Bigr )}{E_{2, \chi_{3,5}}(q^{1/5})}.
\end{align}
Next, rationalize the denominator of \eqref{lev} through the use of Lemma \ref{mos} to achieve
\begin{align} \label{cq1}
   (q&; q)_{\infty}\sum_{n=1}^{\infty} p(5n -1)q^{n/5}  = L_{2,\chi_{3,5}}^{5}(q^{1/5}) \frac{\pi \circ \Omega_{5,0} \Bigl (L_{2,\chi_{3,5}}(q) \Bigr )}{E_{2, \chi_{3,5}}^{5}(q)L_{2,\chi_{3,5}}(q)} \\ &=  \frac{\pi \Bigl ( L_{2,\chi_{3,5}}^{5}(q) \cdot \Omega_{5,0} \Bigl (L_{2,\chi_{3,5}}(q) \Bigr ) \Bigr )}{E_{2, \chi_{3,5}}^{5}(q)L_{2,\chi_{3,5}}(q)} = \frac{\pi \Bigl ( A^{25}(q)B^{25}(q) \cdot \Omega_{5,0} \Bigl (A^{5}(q)B^{5}(q) \Bigr ) \Bigr )}{E_{2, \chi_{3,5}}^{5}(q)L_{2,\chi_{3,5}}(q)}. \label{cq2}
\end{align}
We obtain Equation \eqref{aco} by equating terms on the extreme sides of \eqref{cq1}-\eqref{cq2} of the form $q^{(5n+m)/5}$ for $m$ in each residue class modulo five. In particular, since the argument of $\pi$ in the numerator on the right side of \eqref{cq2} is a homogeneous polynomial in $A^{5}(q)$ and $B^{5}(q)$, we may apply Theorem \ref{quint}
to express the numerator of \eqref{cq2} as a homogeneous polynomial in $A(q)$ and $B(q)$. The action of the operator $\Omega_{5,m}$ may then be deduced from  Theorem \ref{gt}, or equivalently, by extracting terms in the polynomial expansion induced from Theorem \ref{quint} of the form $A^{r}(q)B^{s}(q)$ for $r \equiv m \pmod{5}$. When $m = 0$, we apply \eqref{on1} to derive
\begin{align}
 \nonumber   & (q; q)_{\infty}\sum_{n=0}^{\infty} p(25n +24)q^{n} =   \frac{\Omega_{5,0} \Bigl ( L_{2,\chi_{3,5}}^{5}(q) \cdot \Omega_{5,0} \Bigl (L_{2,\chi_{3,5}}(q) \Bigr ) \Bigr )}{E_{2, \chi_{3,5}}^{5}(q)L_{2,\chi_{3,5}}(q)} 
= 5 \frac{\Omega_{5, 0} \Bigl ( A^{30}(q) B^{30}(q) \Bigr )}{E_{2, \chi_{3,5}}^{5}(q)L_{2,\chi_{3,5}}(q)} \\ & = \frac{25 A(q)^5 B(q)^5}{E_{2, \chi_{3,5}}^{5}(q)} 
\Bigl ( 63 A(q)^{40} -3728 A(q)^{35} B(q)^5 +27861 A(q)^{30} B(q)^{10} \label{lfa} \\ &  \qquad \qquad \qquad  \qquad -25404 A(q)^{25} B(q)^{15}  +21285 A(q)^{20} B(q)^{20} +25404 A(q)^{15} B(q)^{25}  \nonumber \\ & \qquad \qquad \qquad \qquad \qquad  +27861 A(q)^{10} B(q)^{30}  +3728 A(q)^5 B(q)^{35}+63 B(q)^{40} \Bigr ) \nonumber 
\\  &\qquad  = \frac{5^{2}\cdot 63E_{2,\chi_{3,5}}^{4}(q)L_{2, \chi_{3,5}}(q) + 5^{5} \cdot 52E_{2,\chi_{3,5}}^{3}(q)L_{2, \chi_{3,5}}(q)^{2}}{E_{2, \chi_{3,5}}^{5}(q)} \label{fcw}\\ &  \qquad \qquad \qquad + \frac{5^{7} \cdot 63E_{2,\chi_{3,5}}^{2}L_{2, \chi_{3,5}}^{3}(q) + 5^{10} \cdot 6E_{2,\chi_{3,5}}L_{2, \chi_{3,5}}^{4}(q)  + 5^{12}L_{2, \chi_{3,5}}^{5}(q)}{E_{2, \chi_{3,5}}^{5}(q)}. \nonumber
\end{align}
The polynomial in the numerator of \eqref{lfa} is symmetric in absolute value about the middle coefficient. This symmetry is a result of the representations for Eisenstein series from \cite{qeis}. Consequently, we obtain the representation \eqref{fcw} of the numerator of \eqref{lfa} as a homogeneous polynomial in $L_{2, \chi_{3, 5}}(q)$ and $E_{2, \chi_{3,5}}(q)$. To show that Equation \eqref{fcw} is equivalent to Ramanujan's assertion \eqref{bvf1}, apply \eqref{ghee} and Lemma \ref{j5} to obtain
\begin{align}
L_{2, \chi_{3,5}}(q) = q\frac{(q^{5}; q^{5})_{\infty}^{5}}{(q;q)_{\infty}}, \qquad E_{2, \chi_{3,5}}(q) =   \frac{(q;q)_{\infty}^{5}}{(q^{5}; q^{5})_{\infty}}, \qquad  \frac{L_{2, \chi_{3,5}}(q)}{E_{2, \chi_{3, 5}}(q)} = \frac{q(q^{5};q^{5})_{\infty}^{6}}{(q; q)_{\infty}^{6}}. \label{prd}
\end{align}
These calculations induce more general quintic multisections for the partition function.
\begin{thm} \label{nm1}
  Let $\chi(n)$ denote the principal character modulo two, and define
  \begin{align*}
F_{0}(q) = L_{2, \chi_{3, 5}}(q), \qquad F_{n,m}(q) = \Omega_{5, m} \left (L_{2, \chi_{3, 5}}^{5^{n}}(q) \cdot \Omega_{5,0} \Bigl (F_{n-1}(q) \Bigr ) \right ), \qquad n \ge 1,       \end{align*}
and let $\delta_{k}$ be defined by \eqref{par}. Then, for $m = 0, 1, 2, 3, 4$, and $k \in \Bbb N$, $k \ge 2$,
\begin{align}
(q;q)_{\infty}^{5^{\chi(k)}} \sum_{n=0}^{\infty} p(5^{k}n + \delta_{k} +  5^{k-1}m)q^{n} = \frac{F_{k-1,m}(q)}{\prod_{j=0}^{\lfloor k/4 \rfloor} \displaystyle E_{2,\chi_{3,5}}^{5^{2j+1+ \chi(k)}}(q) L_{2, \chi_{3,5}}^{5^{2j+\chi(k)}}(q)}.  \label{bvy} 
\end{align}
\end{thm}

Expansions for the right side of \eqref{bvy} in terms of quintic theta functions may also be derived from recursion formulas appearing in the work of M. Hirschorn and D. Hunt \cite{MR622342}. These authors derive expansions equivalent to those induced from Theorem \ref{nm1} through an industrious use of the latter two relations of Corollary \ref{indus}.

\section{Dissections for generalizations of the partition function} \label{gen_part}
The methods of Section \ref{partition} may be applied to another class of functions considered by Ramanujan \cite[pp. 192-193]{MR1353909}. As before, our approach is based on quintic multisections of Eisenstein series. We first obtain decompositions for a normalized version of the Dedekind eta function $\eta(\tau) = q^{1/24}(q;q)_{\infty}$ in terms of the Rogers-Ramanujan continued fraction. The modular equation \eqref{it} is an immediate consequence of these dissections.
\begin{thm} \label{mil}
  Let $R(q)$ denote the Rogers-Ramanujan continued fraction and define $$(q;q)_{\infty} = \sum_{n=0}^{\infty} p_{-1}(n) q^{n}.$$
Then, for each odd $k \in \Bbb N\setminus \{0\}$, 
\begin{align} \label{vrq}
\sum_{n=0}^{\infty} p_{-1} \Bigl (5^{k} n + \frac{5^{k+1}-1}{24} \Bigr )q^{n} &=  i^{k+1}(q^{5}; q^{5})_{\infty},
\\ \sum_{n=0}^{\infty} p_{-1} \Bigl (5^{k} n + \frac{5^{k+1}-1}{24} - 5^{k-1} \Bigr )q^{n} &= \frac{i^{k-1}q^{1/5}(q^{5}; q^{5})_{\infty}}{R(q)}, \label{vrq1} \\
 \sum_{n=0}^{\infty} p_{-1} \Bigl (5^{k} n + \frac{5^{k+1}-1}{24} + 5^{k-1} \Bigr )q^{n}
&= i^{k+1}q^{-1/5} (q^{5}; q^{5})_{\infty}R(q), 
\end{align}
\begin{align}
 0 &= \sum_{n=0}^{\infty} p_{-1} \Bigl (5^{k} n + \frac{5^{k+1}-1}{24} + 2\cdot 5^{k-1} \Bigr )q^{n} \\ &= \sum_{n=0}^{\infty} p_{-1} \Bigl (5^{k} n + \frac{5^{k+1}-1}{24} + 3\cdot 5^{k-1} \Bigr )q^{n}.
\end{align}
For each even $k \in \Bbb N\setminus\{0\}$, and $m = 0, 1,2, 3, 4$, 
\begin{align} \label{reqf}
  \sum_{n=0}^{\infty} p_{-1} \Bigl (5^{k} n + \frac{5^{k}-1}{24} +  5^{k-1}m \Bigr )q^{n} = \Bigl ( 1 - \chi_{1,5}(m) \Bigr ) \cdot i^{k}(q;q)_{\infty}.
\end{align}
\end{thm}

\begin{proof} We prove the result by induction on $k$. By Lemma \ref{mos} and Theorem \ref{hda}, 
  \begin{align}
    \frac{1}{L_{2,\chi_{3,5}}(q^{1/5})} 
&= \frac{1}{(q;q)_{\infty}^{5}}\sum_{n=-1}^{\infty} p_{-1}(n+1)q^{n/5} \\ &=  \frac{1}{L_{2,\chi_{3,5}}(q^{1/5})} \cdot \frac{E_{2, \chi_{3,5}}^{1/5}(q^{1/5}) L_{2, \chi_{3,5}}^{4/5}(q)}{E_{2, \chi_{3,5}}^{1/5}(q^{1/5}) L_{2, \chi_{3,5}}^{4/5}(q)} \\ &= \frac{L_{2, \chi_{3,5}}^{4/5}(q) E_{2, \chi_{3,5}}^{1/5}(q^{1/5})}{E_{2, \chi_{3,5}}(q) L_{2, \chi_{3,5}}(q)} \\ &= \frac{A^{4}(q)B^{4}(q) \Bigl (B^{2}(q) - A(q) B(q) - A^{2}(q) \Bigr ) }{E_{2, \chi_{3,5}}(q) L_{2, \chi_{3,5}}(q)}. \label{ord}
  \end{align}
Now equate terms with integral powers of $q$ on both sides of \eqref{ord} to derive
\begin{align} \label{do1}
   \frac{1}{(q;q)_{\infty}^{5}}\sum_{n=0}^{\infty} p_{-1}(5n+1)q^{n} &= -\frac{A^{5}(q)B^{5}(q)}{E_{2, \chi_{3,5}}(q) L_{2, \chi_{3,5}}(q)} \\ &= -\frac{L_{2, \chi_{3,5}}(q)}{E_{2, \chi_{3,5}}(q) L_{2, \chi_{3,5}}(q)} = -\frac{(q^{5}; q^{5})_{\infty}}{(q;q)_{\infty}^{5}}. \label{six8}
\end{align}
Equalities \eqref{six8} are a result of \eqref{sp} and \eqref{prd}. The agreement between the extreme sides of \eqref{do1} and \eqref{six8} proves the $k = 1$ case of the equality of \eqref{vrq}. Each initial case in Theorem \ref{mil} is similarly derived by equating terms of the form $q^{r/5}$ for $r$ in each residue class modulo five. Next, divide both side of \eqref{six8} by $E_{2, \chi_{3,5}}(q)$, apply \eqref{prd}, and replace $q$ with $q^{1/5}$ in the resulting identity to conclude
\begin{align} \label{too}
   \frac{1}{(q;q)_{\infty}}\sum_{n=0}^{\infty} p_{-1}(5n+1)q^{n/5} = -1.
\end{align}
The required formulas \eqref{reqf} for the case $k = 2$ may be derived by from \eqref{too} by equating terms of the form $q^{r/5}$ for $r$ in each residue class modulo five. In particular, 
\begin{align} \label{yapd}
  \sum_{n=0}^{\infty} p_{-1}(25n+1)q^{n} = -(q;q)_{\infty} = -\sum_{n=0}^{\infty} p_{-1}(n)q^{n}.
\end{align}
The general case of Theorem \ref{mil} follows from \eqref{yapd} and induction on $k$.
\end{proof}

We now give a new proof of Watson's modular equation \eqref{it} from Theorem \ref{mil}.
\begin{proof}[Proof of \eqref{it}]
  From the quintic decomposition in Theorem \ref{mil}, we deduce
  \begin{align}
    (q;q)_{\infty}  &= \sum_{k=0}^{5} \sum_{n=0}^{\infty} p_{-1}(5n+k) q^{5n +k} \\ &= \frac{q(q^{25}; q^{25})_{\infty}}{R(q^{5})} - q(q^{25}; q^{25})_{\infty} - q(q^{25}; q^{25})_{\infty} R(q^{5}). \label{hwe}
  \end{align}
From \eqref{hwe}, we deduce a famous formula of Ramanujan \cite[p. 238]{ramlost}
\begin{align} \label{zxcv}
  \frac{(q;q)_{\infty}}{q(q^{25}; q^{25})_{\infty}} = \frac{1}{R(q^{5})} - 1 - R(q^{5}) = \frac{B(q^{5})}{A(q^{5})} - 1 - \frac{A(q^{5})}{B(q^{5})}.
\end{align}
From \eqref{zxcv}, we obtain
\begin{align}
  & 5^{2}q^{5}\frac{(q^{25}; q^{25})_{\infty}^{5}}{(q;q)_{\infty}^{5}} +  5^{2}q^{4}\frac{(q^{25}; q^{25})_{\infty}^{4}}{(q;q)_{\infty}^{4}} + 15q^{3}\frac{(q^{25}; q^{25})_{\infty}^{3}}{(q;q)_{\infty}^{3}} +  5q^{2}\frac{(q^{25}; q^{25})_{\infty}^{2}}{(q;q)_{\infty}^{2}} +  q\frac{(q^{25}; q^{25})_{\infty}}{(q;q)_{\infty}}. \nonumber \\ 
&= 
\frac{\left(-3 A\left(q^5\right)^3 B\left(q^5\right)+4 A\left(q^5\right)^2 B\left(q^5\right)^2-2 A\left(q^5\right) B\left(q^5\right)^3+A\left(q^5\right)^4+B\left(q^5\right)^4\right)}
{( B^{2}(q^5) - A(q^5) B(q^5) - A^{2}(q^5))^5} \nonumber \\ &  \times \Bigl (2 A^{4}(q^5) B^{2}(q^5)+4 A^{3}(q^5) B^{3}(q^5)+3 A^{2}(q^5) B^{4}(q^5)+A^{5}(q^5)B(q^{5})+A(q^5)B^{5}(q^5) \Bigr ) \nonumber \\ & \qquad \qquad =  \frac{A^{5}(q) B^{5}(q)}{B^{10}(q) - 11 A^{5}(q) B^{5}(q) - A^{10}(q)}  = \frac{L_{2}(q)}{E_{2, \chi_{3,5}}(q)} = \frac{(q^{5};q^{5})_{\infty}^{6}}{(q; q)_{\infty}^{6}}. \label{bjh}
\end{align}
The last two equalities of \eqref{bjh} follow from identities \eqref{oo1}-\eqref{oo2}. The final equality of \eqref{bjh} follows from the product representations for the weight-two Eisenstein series given by Ramanujan \cite[p. 139]{ramlost}, or by the equivalent formulation from \eqref{prd}.
\end{proof}

The calculations from the proof of the last Theorem may be extended to yield similar expansions for dissections of \eqref{grt} for $k \in \Bbb Z$. The first few cases are given below.
\begin{thm} \label{posi} Let $p_{k}(n)$ be defined by \eqref{grt}. Then
  \begin{align} \label{sda}
\sum_{n=0}^{\infty} p_{2}(5 n +3)q^{n} = 10 &\frac{(q^{5}; q^{5})_{\infty}^{4}}{(q;q)_{\infty}^{6}} + 125 \frac{(q^{5}; q^{5})_{\infty}^{10}}{(q;q)_{\infty}^{12}}, \\ \sum_{n=0}^{\infty} p_{-2}(5n+2)q^{n} &= -(q^{5}; q^{5})_{\infty}(q;q)_{\infty}^{5}, \\ \sum_{n=0}^{\infty} p_{3}(5n+2)q^{n} = 9\frac{(q^{5}; q^{5})_{\infty}^{3}}{(q;q)_{\infty}^{6}} &+ 375\frac{(q^{5}; q^{5})_{\infty}^{9}}{(q;q)_{\infty}^{12}} + 3125 \frac{(q^{5}; q^{5})_{\infty}^{15}}{(q;q)_{\infty}^{18}}, \\ \sum_{n=0}^{\infty} p_{-3}(5n+3)q^{n} &=5(q^{5}; q^{5})_{\infty}(q;q)_{\infty}^{10}, \\ \sum_{n=0}^{\infty} p_{-4}(5n+4)q^{n} &= -5(q^{5}; q^{5})_{\infty}(q;q)_{\infty}^{15}, \\ \sum_{n=0}^{\infty} p_{4}(5n+1)q^{n} = 4\frac{(q^{5}; q^{5})_{\infty}^{2}}{(q;q)_{\infty}^{6}} & +550\frac{(q^{5}; q^{5})_{\infty}^{8}}{(q;q)_{\infty}^{12}} \\ &+12500 \frac{(q^{5}; q^{5})_{\infty}^{14}}{(q;q)_{\infty}^{18}}+78125 \frac{(q^{5}; q^{5})_{\infty}^{20}}{(q;q)_{\infty}^{24}}, \nonumber \\ 
\sum_{n=0}^{\infty} p_{-5}(5n)&q^{n} = \frac{(q;q)_{\infty}^{6}}{(q^{5}; q^{5})_{\infty}}, \\ \sum_{n=0}^{\infty} p_{5}(5n)q^{n} =\frac{(q^{5}; q^{5})_{\infty}}{(q;q)_{\infty}^{6}} + 500&\frac{(q^{5}; q^{5})_{\infty}^{7}}{(q;q)_{\infty}^{12}} + 25000\frac{(q^{5}; q^{5})_{\infty}^{13}}{(q;q)_{\infty}^{18}} \\ &+ 390625\frac{(q^{5}; q^{5})_{\infty}^{19}}{(q;q)_{\infty}^{24}}+1953125\frac{(q^{5}; q^{5})_{\infty}^{25}}{(q;q)_{\infty}^{30}}, \\ 
\sum_{n=0} p_{6}(5n+4)q^{n} &=  \frac{1}{5(q; q)_{\infty}^{5}}\sum_{n=0}^{\infty} p_{1}(25n+24)q^{n}.
  \end{align} 
\end{thm} 
\begin{proof} For each negative index $k$, the claimed dissections for \eqref{gt} are derived as in Theorem \ref{mil}. Expansions for positive indices $k$, may be obtained similarly, as we will demonstrate by formulating \eqref{sda}. Through \eqref{oo2} and  \eqref{prd}, we conclude that
  \begin{align}
\Omega_{5,0} \Bigl (& L_{2, \chi_{3,5}}^{2}(q) \Bigr ) = 5A^{5}(q)B^{5}(q)\Bigl (2B^{5}(q) - A^{5}(q) \Bigr )\Bigl (2A^{5} (q) + B^{5}(q) \Bigr ) \\ &= 125L_{2, \chi_{3,5}}^{2}(q) + 10L_{2, \chi_{3,5}}(q)E_{2,\chi_{3,5}}(q) = (q;q)_{\infty}^{10}\sum_{n=0}^{\infty} p(5n+3)q^{n}. \label{fv1}
  \end{align}
We derive \eqref{sda} by employing the product representation for $E_{2, \chi_{3,5}}(q)$ from \eqref{prd}.
\end{proof}

The congruences induced by Theorem \ref{posi} are special cases of more general results.
\begin{thm} \label{le1} For each $\delta \in \Bbb Z$,
\begin{align} \label{sec2} 
   p_{5\delta - 3} \Bigl ( 5n - 2 \Bigr ) & \equiv 0 \pmod{5}.
\end{align}
\end{thm}
\begin{proof}
From the the binomial theorem and Euler's pentagonal number theorem 
\begin{align}
  \sum_{\mu =0}^{\infty} (-1)^{\mu} (2 \mu +1) q^{\frac{1}{2} \mu (\mu +1)} = (q;q)_{\infty}^{3},
\end{align}
we deduce
  \begin{align}
    \sum_{n=0}^{\infty} p_{5\delta  - 3}(n) q^{n + \lambda} &= (q;q)_{\infty}^{- 5\delta} (q;q)_{\infty}^{3}q^{\lambda} \nonumber  \\ &\equiv (q^{5}; q^{5})_{\infty}^{- \delta}\sum_{\mu =0}^{\infty} (-1)^{\mu} (2 \mu +1) q^{\frac{1}{2} \mu (\mu +1)+ \lambda} \pmod{5}. \label{perf}
  \end{align}
Consider the terms on the right side of \eqref{perf} with exponents of the form
\begin{align} \label{mult1}
  \frac{1}{2} \mu(\mu +1) + \frac{\omega +1}{3} \equiv 0 \pmod{5}.
\end{align}
Theorem \ref{le1} will follow if we can show that for such values of $\mu$, 
\begin{align} \label{mult3}
  5 \mid (2 \mu +1).
\end{align}
Multiply both sides of \eqref{mult1} by $24$ to obtain the equivalent congruence
\begin{align} \label{mult2}
 0 \equiv  -3 (2 \mu +1)^{2} \pmod{5}.
\end{align}
It follows that if \eqref{mult1} is satisfied, then \eqref{mult3} holds. 
\end{proof}
The proof we have given for Theorem \ref{le1} is based upon that given for a complementary result appearing in \cite{berndt_gugg} and given by Theorem \ref{le2}. Together, Theorems \ref{le1} and \ref{le2} include the congruences modulo five obtained from the expansions of Theorem \ref{posi}.
\begin{thm} \label{le2} 
  If $\delta \in \Bbb Z$, $n \in \Bbb N\cup \{0\}$, and $\omega$ is a prime of the form $6 \lambda - 1$, then 
  \begin{align} \label{sec1}
    p_{\delta \omega - 4} \Bigl ( n \omega - \frac{\omega +1}{6} \Bigr ) &\equiv  0 \pmod{\omega}.
  \end{align}
\end{thm}

Certain classes of congruences in Theorems \ref{le1} and \ref{le2} appear to lift to $\Bbb Z_{25}$. 
\begin{conj}
Let $k \in \Bbb N \cup \{0\}$. Then
\begin{align}
  p_{17 + 25k}(5n +3) \equiv 0 \pmod{25}, \qquad p_{11 + 25k}(5n +4) \equiv 0 \pmod{25}.
\end{align}

\end{conj}

By using the techniques of the last section, we may derive congruences for \eqref{gt} modulo $5^{\lambda}$ and generalize Ramanujan's conjectures for the series \eqref{gt}.
\begin{conj}
\begin{align*}
  p_{r}(5^{\lambda}n + \mu_{\lambda}) &\equiv 0 \pmod{5^{\lambda}}, \quad 12\mu_{\lambda} \equiv 1 \pmod{5^{\lambda}}, \quad r \equiv 2 \pmod{5}\\
  p_{k}(5^{\lambda}n + \delta_{\lambda}) &\equiv 0 \pmod{5^{\lambda}}, \quad 24\delta_{\lambda} \equiv 1 \pmod{5^{\lambda}}, \quad k \equiv 1 \pmod{5}.
\end{align*}
\end{conj}

\section{A modular perspective} \label{modular}

To this point, properties of the quintic theta functions have been formulated from an elementary viewpoint.
In this section, we record a number of relevant facts from the perspective of modular forms. We alter our previous notation to refer to the quintic theta functions $A, B, C$, and $D$ as functions of $\tau$, where $q = e^{2 \pi i \tau}$. The reader is referred to \cite{dish} for standard notation and terminology from the theory of modular forms.


The quintic theta functions $A(\tau)$ and $B(\tau)$ are weight-$1/5$ modular forms on $\Gamma(5)$ with appropriate multiplier system \cite{MR1893493}. In addition, they are fundamentally connected to three congruence subgroups of level five with corresponding indices $$\Gamma(5) \subset \Gamma_{1}(5) \subset \Gamma_{0}(5), \qquad [\Gamma_{0}(5): \Gamma_{1}(5)] = 2, \qquad  [\Gamma_{1}(5): \Gamma(5)] = 5. $$

The fifth powers of the quintic theta functions, $A^{5}(\tau)$ and $B^{5}(\tau)$, are shown in \cite{MR1904094} to generate the graded ring of modular forms for $\Gamma_{1}(5)$. Moreover, representations for Eisenstein series on $\Gamma_{0}(5)$ in terms of the generators  $A^{5}(\tau)$ and $B^{5}(\tau)$ for $\Gamma_1(5)$ are symmetric in absolute value about the middle coefficients \cite{huber_quint}. 
In this section we derive a natural correspondence between $A(\tau), B(\tau), C(\tau)$, and $D(\tau)$ by employing modular transformation formulas for the theta constants with characteristic $(\epsilon, \epsilon') \in \Bbb R^{2}$ appearing in the work of H. Farkas and I. Kra \cite{MR1850752}
\begin{align} \label{71}
  \theta \left [ 
    \begin{array}{c}
      \epsilon \\ \epsilon'
    \end{array} \right ](\tau) = e^{\pi i \epsilon \epsilon'/2} \sum_{n=-\infty}^{\infty} e^{\pi i n \epsilon'} q^{\left ( n + \frac{\epsilon}{2} \right )^{2}/2}, \qquad \re \tau > 0,
\end{align}
We may derive theta constant representations for the quintic theta functions from \eqref{71}.
\begin{lem} \label{lem71} Let $A(\tau), B(\tau), C(\tau), D(\tau)$ be defined as in Theorem \ref{mainthm}. Then 
  \begin{align}
    A(\tau) &= -e^{7\pi i/10}q^{-1/40} (q;q)_{\infty}^{-3/5}\cdot  \theta \left [ 
    \begin{array}{c}
      3/5 \\ 1
    \end{array} \right ](5 \tau), \\ \qquad B(\tau) &= e^{-\pi i/10}q^{-1/40} (q;q)_{\infty}^{-3/5} \cdot \theta \left [ 
    \begin{array}{c}
      1/5 \\ 1
    \end{array} \right ](5 \tau), \\ 
 C(\tau) &= \frac{q^{-1/8}}{\sqrt{\frac{1}{2}(5 - \sqrt{5})}} (q^{5};q^{5})_{\infty}^{-3/5}\cdot  \theta \left [ 
    \begin{array}{c}
      1 \\ 3/5
    \end{array} \right ](\tau), \\ 
 D(\tau) &= \frac{q^{-1/8}}{\sqrt{\frac{1}{2}(5 + \sqrt{5})}} (q^{5};q^{5})_{\infty}^{-3/5}\cdot  \theta \left [ 
    \begin{array}{c}
      1 \\ 1/5
    \end{array} \right ](\tau).
  \end{align}
\end{lem}
\begin{thm} \label{tra1} Choose the branch of the fifth root with argument in $[0, 2 \pi/5)$. Then 
  \begin{align}
\label{tw1} A \left (-\frac{1}{5\tau} \right ) &= \sqrt{\frac{1}{2}(5 - \sqrt{5})}  \cdot \frac{e^{-\pi i/10} t^{1/5}}{5^{3/10}} C(\tau), \\  B \left (-\frac{1}{5\tau} \right ) &= \sqrt{\frac{1}{2}(5 + \sqrt{5})} \cdot \frac{e^{- \pi i/10} t^{1/5}}{5^{3/10}} D(\tau). \label{tw2}
  \end{align}
\end{thm}

\begin{proof}
For $\ell, m \in \Bbb Z$, we have \cite[pp. 72--77]{MR1850752}
\begin{align}
   \theta \left [ 
    \begin{array}{c}
      \epsilon \\ \epsilon'
    \end{array} \right ](\tau) = &e^{\mp \epsilon m \pi i} \theta \left [ 
    \begin{array}{c}
      \pm \epsilon + 2 \ell \\ \pm \epsilon' + 2 m
    \end{array} \right ](\tau).
\end{align}
In addition, the theta constants satisfy the modular transformation formula
\begin{align}
  \theta \left [ 
    \begin{array}{c}
      \epsilon \\ \epsilon'
    \end{array} \right ] \left ( \frac{-1}{\tau} \right ) &= e^{-\pi i/4} \sqrt{\tau} e^{\pi i\epsilon \epsilon'/2}  \theta \left [ 
    \begin{array}{c}
      \epsilon' \\ -\epsilon
    \end{array} \right ](\tau),
\end{align}
where we choose the branch of the square root with argument in $[0, \pi)$. We also make use of the transformation formula for the Dedekind $\eta$-function $\eta(\tau) = q^{1/24}(q;q)_{\infty},$
\begin{align}
  \eta \left ( \frac{- 1}{\tau} \right ) = \sqrt{ - i \tau} \eta(\tau), \qquad \im \tau > 0.
\end{align}
Identities \eqref{tw1}-\eqref{tw2} are immediate consequences of these formulas.
\end{proof}

Theorem \ref{tra1} proves Theorem \ref{fricke}, where $\gamma_{1}, \gamma_{2}$ denote the coefficients of $t^{1/5}C(\tau)$ and $t^{1/5}D(\tau)$, respectively, in \eqref{tw1}-\eqref{tw2}. These are roots of the same minimal polynomial over $\Bbb Q$, $1 + 25 x^{10} + 5x^{20}.$ Other consequences of Theorem \ref{tra1} include a variety of representations for the Fricke involution of the Rogers-Ramanujan continued fraction.
\begin{cor} \label{fuls}
If $R(\tau)$ denotes the Rogers-Ramanujan continued fraction, then
  \begin{align} \label{ex2}
    R \left (\frac{-1}{5 \tau} \right )  &= - \beta\frac{C(q)}{D(q)} = \prod_{n=1}^{\infty} \frac{1 + \alpha q^{n} + q^{2n}}{1 + \beta q^{n} + q^{2n}} \\ &=  - \beta \frac{ \sqrt[5]{\frac{(1 + \alpha i) E_{1, \chi_{4, 5}}(q) + (1 - \alpha i)E_{1,\chi_{2, 5}}(q)}{2}}}{\sqrt[5]{\frac{(1 + \beta i) E_{1, \chi_{4, 5}}(q) + (1 - \beta i)E_{1,\chi_{2, 5}}(q)}{2}}}
\\ &=  \frac{ \left ( \frac{1}{\sqrt{R(q)}} \right )^{5} - \left ( \beta \sqrt{R(q)} \right )^{5} }{\left ( \frac{1}{\sqrt{R(q)}} \right )^{5} - \left ( \alpha \sqrt{R(q)} \right )^{5}} = \frac{B(q^{5}) - \beta A(q^{5})}{B(q^{5}) - \alpha A(q^{5})},
  \end{align}
where 
$$\alpha = \frac{1+\sqrt{5}}{2}, \qquad \beta = \frac{1-\sqrt{5}}{2}, \qquad q = e^{2 \pi i \tau}.$$
\end{cor}
Corollary \ref{fuls} may be deduced from Lemma \ref{j5} and Theorems \ref{mainthm}, \ref{nm} and \ref{tra1}. Substituting $\tau/5$ for $\tau$ in Corollary \ref{fuls} yields a wealth of corresponding formulations for the modular transformation formula satisfied by the Rogers-Ramanujan continued fraction. These transformation formulas and their consequences are studied in \cite{MR2133308}.


A. Sebbar proved in \cite{MR1904094} that $\Gamma_{1}(5)$ is one of only six subgroups of the full modular group whose graded ring of modular forms are generated by precisely two elements, each of degree one. As previously mentioned, the generators for $\Gamma_{1}(5)$ are $A^{5}(q)$ and $B^{5}(q)$. The matrix characterization given in Section \ref{s3} for the multisection operator $\Omega_{5,0}$ therefore provides an explicit characterization for the action of the Hecke operator $T_{5}$ on $\Gamma_{1}(5)$ with respect to the basis generated by $A^{5}(q)$ and $B^{5}(q)$. The significance of this formulation for $T_{5}$ may be observed in the last two sections, where $T_{5}$ plays an important role in the multisections for the partition function and its generalizations.

\section{Differential equations for quintic theta functions} \label{s6}
To prove the coupled system of differential equations for $A(q)$ and $B(q)$ given in Theorem \ref{d_quint}, we employ parameterizations for the elliptic parameters from \eqref{ser}, namely $e_{\alpha}(q), P_{\alpha}(q)$, and $Q_{\alpha}(q)$, $\alpha = 1/5, 2/5$ in terms of Eisenstein series. We will use formulas from \cite{qeis} to write the Eisenstein series of weight $k$ in terms of homogeneous polynomials of degree $k$ in $A^{5}(q)$ and $B^{5}(q)$. These calculations allow us to translate Theorem \ref{coupl} into Theorem \ref{d_quint}.

The next lemma demonstrates that the coefficients of the series $e_{1/5}(q)$, $e_{2/5}(q)$, $P_{1/5}(q)$, $P_{2/5}(q)$, $Q_{1/5}(q)$, $Q_{2/5}(q)$ are conjugate elements of the field $\Bbb Q(\sqrt{5})$, and these series are representable in terms of quintic Eisenstein series.
 \begin{lem} \label{cra} Let $e_{\alpha}(q), P_{\alpha}(q)$, and $Q_{\alpha}(q)$ be defined as in \eqref{ser}. Then 
   \begin{align}
\label{fq}     e_{1/5}(q) &= t_{1}(q) + t_{2}(q)  \sqrt{5}, \quad  e_{2/5}(q) = t_{1}(q) - t_{2}(q)  \sqrt{5}, \\ 
  \label{fq1}   P_{1/5}(q) &= t_{3}(q) +  t_{4}(q)\sqrt{5}, \quad  P_{2/5}(q) = t_{3}(q) -  t_{4}(q)\sqrt{5}, \\   \label{fq2} Q_{1/5}(q) &= t_{5}(q) +  t_{6}(q) \sqrt{5}, \quad  Q_{2/5}(q) = t_{5}(q) - t_{6}(q) \sqrt{5},
   \end{align}
where
\begin{align*}
t_{1} &= \frac{1 + 2 i}{2}  E_{1, \chi_{2,5}}(q) + \frac{1 - 2 i}{2} E_{1, \chi_{4,5}}(q), \quad
t_{2} = \frac{E_{1, \chi_{2,5}}(q) - E_{1, \chi_{4,5}}(q)}{2i} ,
\end{align*}
\begin{align*}
  t_{3} = \frac{5}{24}E_{2,\chi_{1,5}}(q) - \frac{1}{4} E_{2,\chi_{3,5}}(q) + \frac{25}{24}E_{2}(q^{5}), \quad
  t_{4} = -\frac{1}{24}E_{2,\chi_{1,5}}(q) + \frac{1}{4} E_{2,\chi_{3,5}}(q) - \frac{5}{24}E_{2}(q^{5}),
\end{align*}
\begin{align*}
  t_{5} = \left ( \frac{1}{2} + \frac{11i}{4} \right )E_{3,\chi_{2,5}}(q) +  \left (\frac{1}{2} - \frac{11i}{4} \right )E_{3,\chi_{4,5}}(q), \quad
  t_{6} = - \frac{5 i }{4} E_{3,\chi_{2,5}}(q) + \frac{5 i}{4} E_{3,\chi_{4,5}}(q).
\end{align*}
 \end{lem}
 \begin{proof}
    The equations in Lemma \ref{cra} follow immediately once we write each identity in the form 
  \begin{align}
    1 + \sum_{n=1}^{\infty} a(n) \frac{q^{n}}{1 - q^{n}}  =1 + \sum_{n=1}^{\infty} b(n) \frac{q^{n}}{1 - q^{n}}.
  \end{align}
To prove the claimed expansion for $P_{1/5}(q)$ on line \eqref{fq1}, we use the fact that
\begin{align}
L(-1, \chi_{1, 5}) = \frac{1}{3}, \qquad  L(-1, \chi_{3, 5}) = -\frac{2}{5}
\end{align}
to deduce from \eqref{ser} and \eqref{eisdef} that
\begin{align*}
  a(n) = - 8 n \sin^{2}(\pi /5) \cos(2 n \pi/5), \quad b(n) = \frac{n}{4} \left ( \frac{10 \beta (20 - 21 \chi_{1, 5}(n))}{\sqrt{5}} + 10 \beta \chi_{3,5}(n) \right ).
\end{align*}
Since $a(n)$ and $b(n)$ are periodic modulo five and agree for $0 \le n \le 4$, we conclude that $a(n) = b(n)$, $n \in \Bbb Z$. The remaining identities of Lemma \ref{cra} are derived similarly.
 \end{proof}

We now rewrite Theorem \ref{coupl} in terms of $t_{i}(q)$, $1 \le i \le 6$. 
\begin{thm} \label{main1}  Let $\dot{t} = q (d/dq)$, and define $t_{i}$, $1 \le i \le 6$ as in Lemma \ref{cra}.
Then
\begin{align*}
  \dot{t_{1}} = \frac{(t_{1} + t_{2})t_{3} + (t_{1} + 5t_{2})t_{4} - t_{5} - t_{6}}{2}, 
 \quad \dot{t_{2}} = \frac{(t_{1} + 5 t_{2})t_{3} + 5(t_{1} + t_{2})t_{4} - t_{5} - 5t_{6}}{10}, 
\end{align*}
\begin{align*}
\dot{t_{3}} = \frac{t_{3}^{2} + 2 t_{3} t_{4} + 5t_{4}^{2} - (t_{1} + 3 t_{2})t_{5} - (t_{1} + 5t_{2})t_{6}}{2}, 
\end{align*}
\begin{align*}
\dot{t_{4}} = \frac{t_{3}^{2} + 10 t_{3} t_{4} + 5t_{4}^{2} - (t_{1} + 5t_{2})t_{5} - 5(t_{1} + 3 t_{2})t_{6}}{10}, 
\end{align*}
\begin{align*}
\dot{t_{5}} &= \frac{3(t_{3} + t_{4} - 5t_{2}^{2})t_{5} - (t_{3} + 5t_{4} + 25t_{2}^{2})t_{6} - (3t_{5} - t_{6})t_{1}^{2} - 2t_{1}t_{2}(6t_{5} + 5t_{6})}{2}, 
\end{align*}
\begin{align*}
\dot{t_{6}} &= \frac{-(t_{3} + 5t_{4} + 25t_{2}^{2})t_{5} + 15(t_{3} + t_{4} - 5t_{2}^{2})t_{6} + (t_{5} - 15t_{6})t_{1}^{2} - 10t_{1}t_{2}(t_{5}+6t_{6})}{10}.
\end{align*} 
\end{thm}

\begin{proof}
To derive the first differential equations given in Theorem \ref{main1}, we apply Lemma \ref{cra} to rewrite the differential equations obtained from Theorem \ref{coupl}. For index values of $1/5$ and $2/5$, we may transcribe \eqref{deqe} in the forms
\begin{align}
  q \frac{d}{dq} e_{1/5} = \frac{4 \alpha}{\sqrt{5}} \left ( e_{1/5} P_{1/5} - Q_{1/5} \right ), \qquad  q \frac{d}{dq} e_{2/5} = -\frac{4 \beta}{\sqrt{5}} \left ( e_{2/5} P_{1/5} - Q_{2/5} \right ),
\end{align}
so that by Theorem \ref{main1}, 
\begin{align} \label{fg1}
  q \frac{d}{dq} (t_{1} + t_{2} \sqrt{5}) = \frac{4 \alpha}{\sqrt{5}} \Bigl ( (t_{1} + t_{2} \sqrt{5})(t_{3} + t_{4} \sqrt{5}) - (t_{5} + t_{6}\sqrt{5}) \Bigr ), \\ 
  q \frac{d}{dq} (t_{1} - t_{2} \sqrt{5}) = -\frac{4 \beta}{\sqrt{5}} \Bigl ( (t_{1} - t_{2} \sqrt{5})(t_{3} - t_{4} \sqrt{5}) - (t_{5} - t_{6}\sqrt{5}) \Bigr ).  \label{fg2}
\end{align}
By adding equations \eqref{fg1} and \eqref{fg2} and dividing the result by $2$, we obtain the first equation of Theorem \ref{main1}. The remaining differential equations of Theorem \ref{main1} are derived in a similar way from Theorem \ref{coupl} and Lemma \ref{cra}. 
\end{proof}


The parameterizations from Theorem \ref{par1} allow us to write the series $t_{i}(q)$, $1 \le i \le 6$, in terms of quintic theta functions and the Eisenstein series of weight two for $SL(2, \Bbb Z)$. 

\begin{lem} \label{paq}
  \begin{align}
    t_{1}(q) &= 2 A^{5}(q) + B^{5}(q), \qquad t_{2}(q) = A^{5}(q), \\ 
    t_{3}(q) &= \frac{25}{24} E_{2}(q^{5}) - \frac{1}{24}B^{10}(q) + \frac{11}{4} A^{5}(q) B^{5}(q) + \frac{11}{24}A^{10}(q), \\
   t_{4}(q) &= - \frac{5}{24}E_{2}(q^{5}) + \frac{5}{24} B^{10}(q) - \frac{11}{4} A^{5}(q) B^{5}(q) - \frac{7}{24} A^{10}(q), \\ 
   t_{5}(q) &= B^{15}(q) - \frac{33}{2} B^{10}(q) A^{5}(q) + \frac{119}{2} B^{5}(q) A^{10}(q) + \frac{11}{2} A^{15}(q), \\ 
   t_{6}(q) &= \frac{5}{2} B^{10}(q) A^{5}(q) - \frac{55}{2} B^{5}(q) A^{10}(q) - \frac{5}{2} A^{15}(q). 
  \end{align}
\end{lem}
We may derive the first two differential equations of Theorem \ref{d_quint} by inserting the parameterizations from Lemma \ref{paq} into the differential equations for $t_{1}(q)$ and $t_{2}(q)$ appearing in Theorem \ref{main1}. By \eqref{g1}, the third equation of Theorem \ref{d_quint} is equivalent to Ramanujan's differential equation for $E_{2}(q)$
\begin{align}
  q\frac{d}{dq} E_{2}(q) = \frac{E_{2}^{2}(q) - E_{4}(q)}{12}. 
\end{align}

One interesting application of Theorem \ref{d_quint} is in finding solutions to  differential equations involving modular forms from M Kaneko's work \cite{japan}.
\begin{cor} \label{bu} 
For $n = 0$, $A(q)$ and $B(q)$ are linearly independent solutions to the differential equation
  \begin{align} \label{dew11}
  f''(q) &- \frac{n+1}{5} E_{2}(q) f'(q) + \frac{(n+1)(6 n +1)}{50} E_{2}'(q) f(q)  = 0. 
\end{align}
\end{cor}
Corollary \ref{bu} follows directly from Theorem \ref{d_quint}.  The method of Frobenius can be used to show that if $t = A^{5}(q)/B^{5}(q)$, $s = -B^{5}(q)/A^{5}(q)$, then formal power series expansions for the solutions to \eqref{dew11} about zero are given by 
\begin{align} \label{ti}
f_{n}(t) = B^{6 n +1}(q) \left ( 1 + \sum_{k=1}^{\infty} a_{k} t^{k}\right ), \qquad g_{n}(s) = A^{6 n +1}(q) \left ( 1 + \sum_{k=1}^{\infty} a_{k} s^{k}\right ),  
\end{align}
If we require $a_{0} = 1$, then, by equating coefficients on each side of \eqref{dew11}, we obtain 
\begin{align}
  a_{1} = \frac{3n(6n +1)}{4 - n},
\end{align}
and, for $k \ge 2$, 
\begin{align} \label{tu}
  a_{k} &= \left ( \frac{55 k^2-11 k (11+6 n)+3 (2+n) (11+6 n)}{k (5 k - n - 1) } \right ) a_{k-1} \\ & \qquad \qquad \qquad \qquad  + \left ( \frac{(-11+5 k-6 n) (-2+k-n) }{k (5 k - n -1) } \right ) a_{k-2}.
\end{align}
Using \eqref{tu}, we may deduce that for $n \le 8$, $a_{k} = 0$ for $ k > n$, so that $f_{n}$ and $g_{n}$ are polynomials of degree at most $n$. Using this fact and \eqref{tu}, we may show that
\begin{align*}
f_{1}(t) &= 1 + 7 t, \quad f_{2}(t) = 1 + 39 t - 26 t^{2}, \quad  f_{3}(t) = 1 + 171t + 247 t^{2} - 57t^{3}, \\  f_{5}(t) &= 1-465 t -10385 t^{2}-2945 t^{3} -8370 t^{4} +682 t^{5}, \\ 
f_{6}(t) &= 1 - 333t -17390 t^{2} -54390 t^{3} +26640 t^{4} -64158 t^{5} +3774 t^{6}, \\  f_{7}(t) &= 1 - 301t - 36421t^2- 310245t^3 + 10535t^4 - 422303t^5 + 283843t^6 -12857t^7, \\ 
f_8(t) &= 1 - 294t - 101528t^2 - 1798692t^3 - 2747430t^4 - 387933t^5
-2086028t^6 \\ & \qquad \qquad + 740544t^7 - 26999t^8.
\end{align*}
M. Kaneko derived these polynomials in \cite{japan} 
using the valence formula from the theory of modular forms. For $n \ge 10$, $n \not \equiv 4 \pmod{5}$, Kaneko used induction to show that solutions to \eqref{dew11} satisfy the recurrence 
\begin{align}
  f_{n}(t) &= (1 + t^{2})(1 - 522t - 10006 t^{2} + 522 t^{3} + t^{4})f_{n-5}(t) \\ & \qquad + 12 \frac{(6n-29)(6n-49)}{(n-4)(n-9)} t (1 - 11 t - t^{2})^{5}f_{n-10}(t).
\end{align}


The differential equations for $A$ and $B$ may also be used to evaluate the Schwarz derivative 
\begin{align}
  \{f,\tau\} = 2 \frac{f'''}{f'} - 3 \left ( \frac{f''}{f'} \right )^{2}, \qquad ' = \frac{d}{d\tau},
\end{align}
for forms parameterized by $A(q)$ and $B(q)$. In particular, when we apply the Schwartz derivative to the canonical hauptmodul $A^{5}/B^{5}$ for $\Gamma_{1}(5)$, we obtain, from Theorem \ref{d_quint},
  $$A^{20} + 12 A^{15}B^{5} + 134A^{10}B^{10} -12A^{5}B^{15} +B^{20},$$
a parameterization from \cite{MR1904094} for the theta function of the $8$-dimensional $5$-modular lattice $Q_8(1)$ with determinant $625$ and minimal norm $4$, also known as the icosian or Maass lattice \cite{MR967491}.


\end{document}